\documentclass[12pt,fullpage]{amsart}

\usepackage{amssymb,amsmath,amsthm,amsfonts}
\usepackage[mathscr]{euscript}
\usepackage{dsfont}
\usepackage{mathrsfs}
\usepackage{esint}
\usepackage{graphicx}
\usepackage{enumitem}
\usepackage{geometry}
\renewcommand{\mathcal}{\mathscr}
\renewcommand{\leq}{\leqslant}

\renewcommand{\geq}{\geqslant}

\usepackage{color}
\definecolor{citation}{rgb}{0.2,0.5,0.2}
\definecolor{formula}{rgb}{0.1,0.2,0.5}
\definecolor{url}{rgb}{0,0.2,0.7}
\usepackage[colorlinks=true,linkcolor=formula,urlcolor=url,citecolor=citation]{hyperref}

\newtheorem{theo}{Theorem}[section]
\newtheorem{cor}[theo]{Corollary}
\newtheorem{lemma}[theo]{Lemma}
\newtheorem{prop}[theo]{Proposition}

\theoremstyle{definition}
\newtheorem{de}[theo]{Definition}
\theoremstyle{fact}
\newtheorem{fact}[theo]{Fact}
\theoremstyle{remark}
\newtheorem{remark}[theo]{Remark}
\newtheorem{ex}[theo]{Example}
\numberwithin{equation}{section}

\renewcommand{\tilde}{\widetilde}

\def\R {\mathbb{R}}

\def\N {\mathbb{N}}

\def\Z {\mathbb{Z}}

\def\eps{\varepsilon}

\newlength{\defbaselineskip}
\newcommand{\setlinespacing}[1]
           {\setlength{\baselineskip}{#1 \defbaselineskip}}
\allowdisplaybreaks[4]

\begin{document}

\title[On restrictions of Besov functions]
{On restrictions of Besov functions}

\author[Julien Brasseur]{Julien Brasseur}
\address[Julien Brasseur]{INRA Avignon, unit\'e BioSP and Aix-Marseille Univ, CNRS,
Centrale Marseille,
I2M, Marseille,
France}
\email{julien.brasseur@univ-amu.fr, julien.brasseur@inra.fr}

\begin{abstract}
In this paper, we study the smoothness of restrictions of Besov functions. It is known that for any $f\in B_{p,q}^s(\mathbb{R}^N)$ with $q\leq p$ we have $f(\cdot,y)\in B_{p,q}^s(\mathbb{R}^d)$ for a.e. $y\in \mathbb{R}^{N-d}$. We prove that this is no longer true when $p<q$. Namely, we construct a function $f\in B_{p,q}^s(\mathbb{R}^N)$ such that $f(\cdot,y)\notin B_{p,q}^s(\mathbb{R}^d)$ for a.e. $y\in \mathbb{R}^{N-d}$. We show that, in fact, $f(\cdot,y)$ belong to $B_{p,q}^{(s,\Psi)}(\mathbb{R}^d)$ for a.e. $y\in\mathbb{R}^{N-d}$, a Besov space of generalized smoothness, and, when $q=\infty$, we find the optimal condition on the function $\Psi$ for this to hold. The natural generalization of these results to Besov spaces of generalized smoothness is also investigated.
\end{abstract}

\subjclass[2010]{35J50}

\keywords{Besov spaces, restriction to almost every hyperplanes, generalized smoothness.}

\maketitle

\tableofcontents

\section{Introduction}

In this paper, we address the following question: given a function $f\in B_{p,q}^s(\R^N)$,
\begin{center}
\emph{what can be said about the smoothness of} $f(\cdot,y)$ \emph{for a.e.} $y\in\R^{N-d}$ \emph{?}
\end{center}
In order to formulate this as a meaningful question, one is naturally led to restrict oneself to $1\leq d<N$, $0<p,q\leq\infty$ and $s>\sigma_p$, where
\begin{align}\sigma_p=N\bigg(\frac{1}{p}-1\bigg)_+, \label{sigmap}
\end{align}
since otherwise $f\in B_{p,q}^s(\R^N)$ need not be a regular distribution.

Let us begin with a simple observation. If $f\in L^p(\R^N)$ for some $0<p\leq\infty$, then
$$ f(\cdot,y)\in L^p(\mathbb{R}^d) \qquad{\mbox{ for a.e. }}~~y\in\mathbb{R}^{N-d}. $$
This is a straightforward consequence of Fubini's theorem. Using similar Fubini-type arguments, one can show that, if $f\in W^{s,p}(\R^N)$ for some $0<p\leq\infty$ and $\sigma_p<s\notin\N$, then we have $f(\cdot,y)\in W^{s,p}(\mathbb{R}^d)$ for a.e. $y\in\mathbb{R}^{N-d}$. We say that these spaces have the \emph{restriction property}.

Unlike their cousins, the Triebel-Lizorkin spaces $F_{p,q}^s(\mathbb{R}^N)$, Besov spaces do not enjoy the Fubini property unless $p=q$, that is
\begin{align}
&\sum_{j=1}^N\left\| \|f(x_1,...,x_{j-1},\cdot,x_{j+1},...,x_N)\|_{B_{p,q}^s(\mathbb{R})}\right\|_{L^p(\mathbb{R}^{N-1})}, \nonumber
\end{align}
is an equivalent quasi-norm on $B_{p,q}^s(\R^N)$ if, and only if, $p=q$; while the counterpart for $F_{p,q}^s(\R^N)$ holds for any given values of $p$ and $q$ where it makes sense (see \cite[Theorem 4.4, p.36]{Trieb} for a proof). In particular, $B_{p,p}^s(\R^N)$ and $F_{p,q}^s(\R^N)$ have the restriction property. It is natural to ask wether or not this feature holds in $B_{p,q}^s(\R^N)$ for an arbitrary $q\neq p$.

Let us recall some known facts.
\begin{fact}\label{FACT1}
Let $N\geq2$, $1\leq d<N$, $0<q\leq p\leq\infty$, $s>\sigma_p$ and $f\in B_{p,q}^s(\mathbb{R}^N)$. Then,
\begin{align}
f(\cdot,y)\in B_{p,q}^{s}(\R^d) \qquad{\mbox{ for a.e. }}~~y\in\R^{N-d}. \nonumber
\end{align}
\end{fact}
(A proof of a slightly more general result will be given in the sequel, see Proposition \ref{caracBesov2}.)

In fact, there is a weaker version of Fact \ref{FACT1}, which shows that this stays "almost" true when $p<q$. This can be stated as follows
\begin{fact}\label{FACT2}
Let $N\geq2$, $1\leq d<N$, $0<p<q\leq\infty$, $s>\sigma_p$ and $f\in B_{p,q}^s(\mathbb{R}^N)$. Then,
\begin{align}
f(\cdot,y)\in\bigcap_{s'<s} B_{p,q}^{s'}(\R^d) \qquad{\mbox{ for a.e. }}~~y\in\R^{N-d}. \nonumber
\end{align}
\end{fact}
See e.g. \cite[Theorem 1]{Jaffard} or \cite[Theorem 1.1]{Aubry}.

Mironescu \cite{Petru} suggested that it might be possible to construct a counterexample to Fact \ref{FACT1} when $p<q$. We prove that this is indeed the case. This is quite remarkable since, to our knowledge, the list of properties of the spaces $B_{p,q}^s$ where $q$ plays a crucial role is rather short.

Our first result is the following
\begin{theo}\label{THEOREM}
Let $N\geq2$, $1\leq d<N$, $0< p<q\leq\infty$ and let $s>\sigma_p$. Then, there exists a function $f\in B_{p,q}^s(\R^N)$ such that
\begin{align}
f(\cdot,y)\notin B_{p,\infty}^{s}(\R^d) \qquad{\mbox{ for a.e. }}~~y\in \R^{N-d}. \nonumber
\end{align}
\end{theo}
Note that this is actually stronger than what we initially asked for, since $B_{p,q}^s\hookrightarrow B_{p,\infty}^s$.
\begin{remark}
We were informed that, concomitant to our work, a version of Theorem \ref{THEOREM} for $N=2$ and $p\geq1$ was proved by Mironescu, Russ and Sire in \cite{MRS}. We present another proof independent of it with different techniques. In fact, we will even prove a generalized version of Theorem \ref{THEOREM} that incorporates other related function spaces (see Theorem \ref{THEOREM2}) which is of independent interest.
\end{remark}
Despite the negative conclusion of Theorem \ref{THEOREM}, one may ask if something weaker than Fact \ref{FACT1} still holds when $p<q$. For example, by standard embeddings, we know that
$$ B_{p,q}^s(\R^N)\hookrightarrow A^{s,p}(\R^N) \qquad{\mbox{ for any }}~~0<q<\infty, $$
where $A^{s,p}(\R^N)$ stands for respectively
\begin{align}
C^{s-\frac{N}{p}}(\R^N),\,\,\,\mathrm{BMO}(\R^N)\,\,\,\text{and }\,L^{\frac{Np}{N-sp},\infty}(\R^N), \label{Asp}
\end{align}
when respectively $sp>N$, $sp=N$ and $sp<N$ (see Subsection \ref{SU:spEM}). In particular, we may infer from Fact \ref{FACT1} that if $q\leq p$, then for every $f\in B_{p,q}^s(\R^N)$ it holds
$$ f(\cdot,y)\in A^{s,p}(\R^d) \qquad{\mbox{ for a.e. }}~~y\in\R^{N-d}.$$

It is tempting to ask wether the same is true when $p<q$. But, as it turns out, even this fails to hold. This is the content of our next result.
\begin{theo}\label{TH:BMO}
Let $N\geq2$, $1\leq d<N$, $0<p<q\leq\infty$ and let $s>\sigma_p$. Then, there exists a function $f\in B_{p,q}^s(\R^N)$ such that 
\begin{align*}
f(\cdot,y)\notin A^{s,p}(\R^d) \qquad{\mbox{ for a.e. }}~~y\in\R^{N-d}.
\end{align*}
\end{theo}
It is nonetheless possible to refine the conclusions of Fact \ref{FACT2} and Theorem \ref{THEOREM}. We find that a natural way to characterize such restrictions is to look at a more general scale of functions known as \emph{Besov spaces of generalized smoothness}, denoted by $B_{p,q}^{(s,\Psi)}(\R^N)$ (see Definition \ref{G:BESOV}). This type of spaces was first introduced by the Russian school in the mid-seventies (see e.g. \cite{Kalyabin,Kudry,Lizorkin}) and was shown to be useful in various problems ranging from Black-Scholes equations \cite{Schneider} to the study of pseudo-differential operators \cite{Almeida,Jacob,Leopold1,Leopold}. A comprehensive state of art covering both old and recent material can be found in \cite{Farkas}. Several versions of these spaces were studied in the literature, from different points of view and different degrees of generality. We choose to follow the point of view initiated by Edmunds and Triebel in \cite{Edmunds} (see also \cite{Cobos,ET,LeopoldF,Mou,Trieb}), which seems better suited to our purposes. Here, $s$ remains the dominant smoothness parameter and $\Psi$ is a positive function of log-type called \emph{admissible} (see Definition \ref{DE:adm}). That admissible function is a finer tuning that allows encoding more general types of smoothness. The simplest example is the function $\Psi\equiv1$ for which one has $B_{p,q}^{(s,\Psi)}(\R^N)=B_{p,q}^s(\R^N)$.

More generally, the spaces $B_{p,q}^{(s,\Psi)}(\R^N)$ are intercalated scales between $B_{p,q}^{s-\eps}(\R^N)$ and $B_{p,q}^{s+\eps}(\R^N)$. For example: if $\Psi$ is increasing, then we have
$$ B_{p,q}^s(\R^N)\hookrightarrow B_{p,q}^{(s,\Psi)}(\R^N)\hookrightarrow B_{p,q}^{s'}(\R^N) \qquad{\mbox{ for every }}~~s'<s, $$
see \cite[Proposition 1.9(vi)]{Mou}.

We prove that restrictions of Besov functions to almost every hyperplanes belong to $B_{p,q}^{(s,\Psi)}(\R^d)$, whenever $\Psi$ satisfies the following growth condition
\begin{align}
\sum_{j\geq0}\Psi(2^{-j})^\chi<\infty, \label{H:WEIGHT}
\end{align}
with $\chi=\frac{qp}{q-p}$ (resp. $\chi=p$ if $q=\infty$).

More precisely, we prove the following
\begin{theo}\label{TH:WEIGHT}
Let $N\geq2$, $1\leq d<N$, $0<p<q\leq\infty$, $s>\sigma_p$ and let $\Psi$ be an admissible function satisfying \eqref{H:WEIGHT}.
Suppose that $f\in B_{p,q}^s(\R^N)$. Then,
$$ f(\cdot,y)\in B_{p,q}^{(s,\Psi)}(\R^d) \qquad{\mbox{ for a.e. }}~~y\in \R^{N-d}. $$
\end{theo}
It turns out the condition \eqref{H:WEIGHT} on $\Psi$ in Theorem \ref{TH:WEIGHT} is optimal, at least when $q=\infty$. In other words, we obtain a sharp characterization of the aforementioned loss of regularity.
\begin{theo}\label{TH:WEIGHT33}
Let $N\geq2$, $1\leq d<N$, $0<p<q\leq\infty$, $s>\sigma_p$ and let $\Psi$ be an admissible function that does not satisfy \eqref{H:WEIGHT}. If $q<\infty$ and $\Psi$ is increasing suppose, in addition, that
\begin{align}
\frac{qp}{q-p}<\frac{1}{c_\infty} \qquad{\mbox{ where }}~~c_\infty:=\sup_{0<t\leq1}\,\log_2\frac{\Psi(t)}{\Psi(t^2)}. \label{pqH}
\end{align}
Then, there is a function $f\in B_{p,q}^s(\R^N)$ such that
$$ f(\cdot,y)\notin B_{p,q}^{(s,\Psi)}(\R^d) \qquad{\mbox{ for a.e. }}~~y\in\R^{N-d}. $$
\end{theo}
\begin{remark}
Notice that condition \eqref{pqH} is sufficient and also not far from being necessary to ensure that \eqref{H:WEIGHT} does not hold, as it happens that for some particular choices of $\Psi$, \eqref{H:WEIGHT} is equivalent to $\frac{qp}{q-p}>\frac{1}{c_\infty}$.
\end{remark}

A fine consequence of Theorem \ref{TH:WEIGHT} is that it provides a substitute for $A^{s,p}(\R^d)$ when $p<q$ (in Theorem \ref{TH:BMO}), which could be of interest in some applications (see e.g. \cite{BBM,MRS}). For example, if $sp>d$, $p<q$ and \eqref{H:WEIGHT} is satisfied, then by Theorem \ref{TH:WEIGHT} and \cite[Proposition 3.4]{Caetano} we have
$$ \forall f\in B_{p,q}^s(\R^N),\,\,\,\,f(\cdot,y)\in C^{(s-\frac{d}{p},\Psi)}(\R^d) \qquad{\mbox{ for a.e. }}~~y\in\R^{N-d},$$
where $C^{(\alpha,\Psi)}(\R^d)$ is the generalized H\"older space $B_{\infty,\infty}^{(\alpha,\Psi)}(\R^d)$ (see Remark \ref{WEIGHT:EMB} below).

\begin{remark}
It is actually possible to formulate Theorems \ref{TH:WEIGHT} and \ref{TH:WEIGHT33} in terms of the space $B_{p,q}^{w(\cdot)}(\R^d)$ introduced by Ansorena and Blasco in \cite{Ansorena2,Ansorena}, even though their results do not allow to handle higher orders $s\geq1$ and neither the case $0<p<1$ nor $0<q<1$. Nevertheless, this is merely another side of the same coin and we wish to avoid unnecessary complications. Beyond technical matters, our approach is motivated by the relevance of the scale $B_{p,q}^{(s,\Psi)}(\R^d)$ in physical problems and in fractal geometry (see e.g. \cite{Edmunds,ET,Mou,TriebSpectra,Trieb}).
\end{remark}

In the course of the paper we will also address the corresponding problem with $f\in B_{p,q}^{(s,\Psi)}(\R^N)$ instead of $f\in B_{p,q}^s(\R^N)$ which is of independent interest. In fact, as we will show, our techniques allow to extend Theorems \ref{THEOREM}, \ref{TH:WEIGHT} and \ref{TH:WEIGHT33} to this generalized setting with almost no modifications, see Theorems \ref{THEOREM2}, \ref{TH2:WEIGHT}, \ref{TH3:WEIGHT} and Remark \ref{RE:BGene}. \\


The paper is organized as follows. In the forthcoming Section \ref{SE:notations} we recall some useful definitions and results related to Besov spaces. In Section \ref{SE:SUITE}, we give some preliminary results on sequences which will be needed for our purposes. In Section \ref{SE:ESTIMATE}, we establish some general estimates within the framework of subatomic decompositions and, in Section \ref{SE:FACILE}, we use these estimates to prove a generalization of Fact \ref{FACT1} which will be used to prove Theorem \ref{TH:WEIGHT}. In Section \ref{SE:DIFFICILE}, we prove at a stroke Theorems \ref{THEOREM} and \ref{TH:BMO} using the results collected at Section \ref{SE:SUITE}. Finally, in Section \ref{SE:WEIGHT}, we prove Theorems \ref{TH:WEIGHT} and \ref{TH:WEIGHT33}.

\section{Notations and definitions}\label{SE:notations}

For the convenience of the reader, we specify below some notations used all along this paper.

As usual, $\R$ denotes the set of all real numbers, $\mathbb{C}$ the set of all complex numbers and $\Z$ the collection of all integers. The set of all \emph{nonnegative} integers $\{0,1,2,...\}$ will be denoted by $\N$, and the set of all \emph{positive} integers $\{1,2,...\}$ will be denoted by $\N^*$.

The $N$-dimensional real Euclidian space will be denoted by $\R^N$. Similarly, $\N^N$ (resp. $\Z^N$) stands for the lattice of all points $m=(m_1,...,m_N)\in\R^N$ with $m_j\in\N$ (resp. $m_j\in\Z$).

Given a real number $x\in\R$ we denote by $\lfloor x\rfloor$ its integral part and by $x_+$ its positive part $\max\{0,x\}$. By analogy, we write $\R_+:=\{x_+:x\in\R\}$.

The cardinal of a discrete set $E\subset\Z$ will be denoted by $\# E$. Given two integers $a,b\in\Z$ with $a<b$ we denote by $[\![a,b]\!]$ the set of all integers belonging to the segment line $[a,b]$, namely
$$[\![a,b]\!]:=[a,b]\cap\Z.$$ 

We will sometimes make use of the approximatively-less-than symbol "$\lesssim$", that is we write $a\lesssim b$ for $a\leq C\,b$ where $C>0$ is a constant independent of $a$ and $b$. Similarly, $a\gtrsim b$ means that $b\lesssim a$. Also, we write $a\sim b$ whenever $a\lesssim b$ and $b\lesssim a$.

We will denote by $\mathcal{L}_N$ the $N$-dimensional Lebesgue measure and by $B_R$ the $N$-dimensional ball of radius $R>0$ centered at zero.

The characteristic function of a set $E\subset\R^N$ will be denoted by $\mathds{1}_E$.

We recall that a \emph{quasi-norm} is similar to a norm in that it satisfies the norm axioms, except that the triangle inequality is replaced by
$$\|x+y\|\leq K(\|x\|+\|y\|), \hbox{ for some } K>0.$$

Given two quasi-normed spaces $(A,\left\|\cdot\right\|_A)$ and $(B,\left\|\cdot\right\|_B)$, we say that $A\hookrightarrow B$ when $A\subset B$ with continuous embedding, i.e. when
$$ \|f\|_B\lesssim \|f\|_A, \hbox{ for all } f\in A.$$

Further, we denote by $\ell^p(\N)$, $0<p<\infty$, the space of sequences $u=(u_j)_{j\geq0}$ such that
$$ \|u\|_{\ell^p(\N)}:=\bigg(\sum_{j\geq0}|u_j|^p\bigg)^{1/p}<\infty, $$
and by $\ell^\infty(\N)$ the space of bounded sequences.

As usual, $\mathcal{S}(\R^N)$ denotes the (Schwartz) space of rapidly decaying functions and $\mathcal{S}'(\R^N)$ its dual, the space of tempered distributions.

Given $0<p\leq\infty$, we denote by $L^p(\R^N)$ the space of measurable functions $f$ in $\R^N$ for which the $p$-th power of the absolute value is Lebesgue integrable (resp. $f$ is essentially bounded when $p=\infty$), endowed with the quasi-norm
$$ \|f\|_{L^p(\R^N)}:=\left(\int_{\R^N}|f(x)|^p\mathrm{d}x\right)^{1/p}, $$
(resp. the essential sup-norm when $p=\infty$).

We collect below the different representations of Besov spaces which will be in use in this paper.

\subsection{Classical Besov spaces}
Perhaps the simplest (and the most intuitive) way to define Besov spaces is through finite differences. This can be done as follows.

Let $f$ be a function in $\mathbb{R}^N$. Given $M\in\mathbb{N}^\ast$ and $h\in\mathbb{R}^N$, let
\[\Delta_h^Mf(x)=\sum_{j=0}^M(-1)^{M-j}\binom{M}{j}f(x+hj),\]
be the iterated difference operator.

Within these notations, Besov spaces can be defined as follows.
\begin{de}\label{DEFBESOV}
Let $M\in\N^\ast$, $0<p,q\leq\infty$ and $s\in(0,M)$ with $s>\sigma_p$ where $\sigma_p$ is given by \eqref{sigmap}. The \emph{Besov space} $B_{p,q}^s(\mathbb{R}^N)$ consists of all functions $f\in L^p(\mathbb{R}^N)$ such that
\begin{align*}
[f]_{B_{p,q}^s(\mathbb{R}^N)}:=\left(\int_{|h|\leq1}\|\Delta_h^Mf\|_{L^p(\R^N)}^q\frac{\mathrm{d}h}{|h|^{N+sq}}\right)^{1/q}<\infty, 
\end{align*}
which, in the case $q=\infty$, is to be understood as
\[[f]_{B_{p,\infty}^s(\mathbb{R}^N)}:=\sup_{|h|\leq1}\frac{\|\Delta_h^Mf\|_{L^p(\mathbb{R}^N)}}{|h|^s}<\infty.\]
\end{de}
\noindent The space $B_{p,q}^s(\mathbb{R}^N)$ is naturally endowed with the quasi-norm
\begin{align}
\|f\|_{B_{p,q}^s(\mathbb{R}^N)}:=\|f\|_{L^p(\mathbb{R}^N)}+[f]_{B_{p,q}^s(\mathbb{R}^N)}. \label{NORM:FD}
\end{align}
\begin{remark}
Different choices of $M$ in \eqref{NORM:FD} yield equivalent quasi-norms.
\end{remark}
\begin{remark}
If $p,q\geq1$, then $\left\|\cdot\right\|_{B_{p,q}^s(\R^N)}$ is a norm. However, if either $0<p<1$ or $0<q<1$, then the triangle inequality is no longer satisfied and it is only a quasi-norm. Nevertheless, we have the following useful inequality
$$ \|f+g\|_{B_{p,q}^s(\R^N)}\leq \left(\|f\|_{B_{p,q}^s(\R^N)}^\eta+\|g\|_{B_{p,q}^s(\R^N)}^\eta\right)^{1/\eta}, $$
where $\eta:=\min\{1,p,q\}$, which compensates the absence of a triangle inequality.
\end{remark}
For our purposes, we shall require a more abstract apparatus which will be provided by the so-called \emph{subatomic} (or \emph{quarkonial}) \emph{decompositions}. This provides a way to decompose any $f\in B_{p,q}^s(\R^N)$ along elementary building blocks (essentially made up of a single function independent of $f$) and to, somehow, reduce it to a sequence of numbers (depending linearly on $f$). This type of decomposition first appeared in the monograph \cite{TriebSpectra} of Triebel and was further developed in \cite{Trieb} (see also \cite{Izuki,Knopova,Nakamura,Trieb06}). We outline below the basics of the theory.

Given $\nu\in\N$ and $m\in\Z^N$, we denote by $Q_{\nu,m}\subset\R^N$ the cube with sides parallel to the coordinate axis, centered at $2^{-\nu}m$ and with side-length $2^{-\nu}$.
\begin{de}\label{quarks}
Let $\psi\in C^\infty(\mathbb{R}^N)$ be a nonnegative function with
\[\mathrm{supp}(\psi)\subset\{y\in\mathbb{R}^N:|y|<2^r\},\]
for some $r\geq0$ and
\begin{align*}
\sum_{k\in\mathbb{Z}^N}\psi(x-k)=1 \qquad{\mbox{ for any }}~~x\in\R^N. 
\end{align*}
Let $s>0$, $0<p\leq\infty$, $\beta\in\mathbb{N}^N$ and $\psi^\beta(x)=x^\beta\psi(x)$. Then, for $\nu\in\mathbb{N}$ and $m\in\mathbb{Z}^N$, the function
\begin{align}
(\beta \mathrm{qu})_{\nu,m}(x):=2^{-\nu(s-\frac{N}{p})}\psi^\beta(2^\nu x-m) \qquad{\mbox{ for }}~~x\in\mathbb{R}^N, \label{C:quark}
\end{align}
is called an $(s,p)$-$\beta$-quark relative to the cube $Q_{\nu,m}$.
\end{de}
\begin{remark}\label{infiniQuark}
When $p=\infty$, \eqref{C:quark} means $(\beta \mathrm{qu})_{\nu,m}(x):=2^{-\nu s}\psi^\beta(2^\nu x-m)$.
\end{remark}

\begin{de}\label{bpqsequence}
Given $0<p,q\leq\infty$, we define $b_{p,q}$ as the space of sequences $\lambda=(\lambda_{\nu,m})_{\nu\geq0,m\in\mathbb{Z}^N}$ such that
\[\|\lambda\|_{b_{p,q}}:=\bigg(\sum_{\nu=0}^{\infty}\bigg(\sum_{m\in\mathbb{Z}^N}|\lambda_{\nu,m}|^p\bigg)^{q/p}\bigg)^{1/q}<\infty.\]
\end{de}
\noindent For the sake of convenience we will make use of the following notations
\begin{align*}
\lambda=\{\lambda^\beta:\beta\in\N^N\} \quad\text{with}\quad \lambda^\beta&=\{\lambda_{\nu,m}^\beta\in\mathbb{C}:(\nu,m)\in\mathbb{N}\times\mathbb{Z}^N\}.
\end{align*}
Then, we have the
\begin{theo}\label{BspqQuark}
Let $0<p,q\leq\infty$, $s>\sigma_p$ and $(\beta \mathrm{qu})_{\nu,m}$ be $(s,p)$-$\beta$-quarks according to Definition \ref{quarks}. Let $\varrho>r$ (where $r$ has the same meaning as in Definition \ref{quarks}). Then, $B_{p,q}^s(\mathbb{R}^N)$ coincides with the collection of all $f\in\mathscr{S}'(\mathbb{R}^N)$ which can be represented as
\begin{align}
f(x)=\sum_{\beta\in\mathbb{N}^N}\sum_{\nu=0}^{\infty}\sum_{m\in\mathbb{Z}^N}\lambda_{\nu,m}^\beta(\beta \mathrm{qu})_{\nu,m}(x), \label{dekompp}
\end{align}
where $\lambda^\beta\in b_{p,q}$ is a sequence such that
\begin{align*}
\|\lambda\|_{b_{p,q},\varrho}:=\sup_{\beta\in\mathbb{N}^N}2^{\varrho|\beta|}\|\lambda^\beta\|_{b_{p,q}}<\infty. 
\end{align*}
Moreover,
\begin{align}
\|f\|_{B_{p,q}^s(\mathbb{R}^N)}\sim \inf_{\eqref{dekompp}}\|\lambda\|_{b_{p,q},\varrho}\hspace{0.1em}, \label{ekiv55}
\end{align}
where the infimum is taken over all admissible representations \eqref{dekompp}. In addition, the right hand side of \eqref{ekiv55} is independent of the choice of $\psi$ and $\varrho>r$.
\end{theo}
An elegant proof of this result can be found in \cite[Section 14.15, pp.101-104]{TriebSpectra} (see also \cite[Theorem 2.9, p.15]{Trieb}).
\begin{remark}\label{optdecsub}
It is known that, given $f\in B_{p,q}^s(\mathbb{R}^N)$ and a fixed $\varrho>r$, there is a decomposition $\lambda_{\nu,m}^\beta$ (depending on the choice of $(\beta\mathrm{qu})_{\nu,m}$ and $\varrho$) realizing the infimum in \eqref{ekiv55} and which is said to be an \emph{optimal subatomic decomposition} of $f$. We refer to \cite{Trieb} (especially Corollary 2.12 on p.23) for further details.
\end{remark}

\subsection{Besov spaces of generalized smoothness} Before we define what we mean by "Besov space of generalized smoothness", we first introduce some necessary definitions.
\begin{de}\label{DE:adm}
A real function $\Psi$ on the interval $(0,1]$ is called \emph{admissible} if it is positive and monotone on $(0,1]$, and if
$$ \Psi(2^{-j})\sim \Psi(2^{-2j}) \qquad{\mbox{ for any }}~~j\in\N. $$
\end{de}
\begin{ex}
Let $0<c<1$ and $b\in\R$. Then,
$$ \Psi(x):=|\log_2(cx)|^b \qquad{\mbox{ for }}~~x\in(0,1], $$
is an example of admissible function. Another example is
$$ \Psi(x):=(\log_2|\log_2(cx)|)^b \qquad{\mbox{ for }}~~x\in(0,1]. $$
Roughly speaking, admissible functions are functions having at most logarithmic growth or decay near zero. They may be seen as particular cases of a class of functions introduced by Karamata in the mid-thirties \cite{Karamata1,Karamata2} known as \emph{slowly varying functions} (see e.g. \cite[Definition 1.2.1, p.6]{Bingham} and also \cite[Section 3, p.226]{Bricchi} where the reader may find a discussion on how these two notions relate to each other as well as various examples and references).
\end{ex}
We refer the interested reader to \cite{Mou,Trieb} for a detailed review of the properties of admissible functions.
\begin{de}\label{G:BESOV}
Let $M\in\N^\ast$, $0<p,q\leq\infty$, $s\in(0,M)$ with $s>\sigma_p$ and let $\Psi$ be an admissible function. The \emph{Besov space of generalized smoothness} $B_{p,q}^{(s,\Psi)}(\mathbb{R}^N)$ consists of all functions $f\in L^p(\mathbb{R}^N)$ such that
\begin{align*}
[f]_{B_{p,q}^{(s,\Psi)}(\mathbb{R}^N)}:=\left(\int_{0}^1\,\sup_{|h|\leq t}\|\Delta_h^Mf\|_{L^p(\R^N)}^q\frac{\Psi(t)^q}{t^{1+sq}}\,\mathrm{d}t\right)^{1/q}<\infty, 
\end{align*}
which, in the case $q=\infty$, is to be understood as
\[[f]_{B_{p,\infty}^{(s,\Psi)}(\mathbb{R}^N)}:=\sup_{0<t\leq1}t^{-s}\Psi(t)\sup_{|h|\leq t}\|\Delta_h^Mf\|_{L^p(\mathbb{R}^N)}<\infty.\]
\end{de}
\noindent The space $B_{p,q}^{(s,\Psi)}(\mathbb{R}^N)$ is naturally endowed with the quasi-norm
\begin{align}
\|f\|_{B_{p,q}^{(s,\Psi)}(\mathbb{R}^N)}:=\|f\|_{L^p(\mathbb{R}^N)}+[f]_{B_{p,q}^{(s,\Psi)}(\mathbb{R}^N)}. \label{G:NORM:FD}
\end{align}
\begin{remark}\label{G:M}
Different choices of $M$ in \eqref{G:NORM:FD} yield equivalent quasi-norms.
\end{remark}
\begin{remark}
Observe that, by taking $\Psi\equiv1$, we recover the usual Besov spaces, that is we have
$$ \|f\|_{B_{p,q}^{(s,1)}(\R^N)}\sim\|f\|_{B_{p,q}^s(\R^N)}, $$
see \cite[Theorem 2.5.12, p.110]{Triebel} for a proof of this.
\end{remark}
\begin{remark}
As already mentioned in the introduction, these spaces were introduced by Triebel and Edmunds in \cite{Edmunds,ET} to study some fractal pseudo-differential operators, but the first comprehensive studies go back to Moura \cite{Mou} (see also  \cite{Bricchi2,Caetano,Caetano2,Haroske,Knopova,Trieb06,Trieb10}). In the literature these spaces are usually defined from the Fourier-analytical point of view (e.g. in \cite{Mou,Trieb}) but, as shown in \cite[Theorem 2.5, p.161]{Haroske}, the two approaches are equivalent.
\end{remark}
\begin{remark}
Notice that, here as well, the triangle inequality fails to hold when either $0<p<1$ or $0<q<1$, but, in virtue of the Aoki-Rolewicz lemma, we have the same kind of compensation as in the classical case, see \cite[Lemma 1.1, p.3]{Kalton}. That is, there exists $\eta\in(0,1]$ and an equivalent quasi-norm $\left\|\cdot\right\|_{B_{p,q}^{(s,\Psi)}(\R^N),*}$ with
$$ \|f+g\|_{B_{p,q}^{(s,\Psi)}(\R^N),*}\leq \left(\|f\|_{B_{p,q}^{(s,\Psi)}(\R^N),*}^\eta+\|g\|_{B_{p,q}^{(s,\Psi)}(\R^N),*}^\eta\right)^{1/\eta}. $$
\end{remark}
A fine property of these spaces is that they admit subatomic decompositions. In fact, it suffices to modify the definition of $(s,p)$-$\beta$-quarks to this generalized setting in the following way.
\begin{de}\label{G:quarks}
Let $r$, $\psi$ and $\psi^\beta$ with $\beta\in\N^N$ be as in Definition \ref{quarks}. Let $s>0$ and $0<p\leq\infty$. Let $\Psi$ be an admissible function. Then, in generalization of \eqref{C:quark},
$$ (\beta\mathrm{qu})_{\nu,m}(x):=2^{-\nu(s-\frac{N}{p})}\Psi(2^{-\nu})^{-1}\psi^\beta(2^\nu x-m) \qquad{\mbox{ for }}~~x\in\R^N, $$
is called an $(s,p,\Psi)$-$\beta$-quark.
\end{de}
Then, we have the following
\begin{theo}\label{G:BspqQuark}
Let $0<p,q\leq\infty$, $s>\sigma_p$ and $\Psi$ be an admissible function. Let $(\beta \mathrm{qu})_{\nu,m}$ be $(s,p,\Psi)$-$\beta$-quarks according to Definition \ref{G:quarks}. Let $\varrho>r$ (where $r$ has the same meaning as in Definition \ref{G:quarks}). Then, $B_{p,q}^{(s,\Psi)}(\mathbb{R}^N)$ coincides with the collection of all $f\in\mathscr{S}'(\mathbb{R}^N)$ which can be represented as
\begin{align}
f(x)=\sum_{\beta\in\mathbb{N}^N}\sum_{\nu=0}^{\infty}\sum_{m\in\mathbb{Z}^N}\lambda_{\nu,m}^\beta(\beta \mathrm{qu})_{\nu,m}(x), \label{G:dekompp}
\end{align}
where $\lambda^\beta\in b_{p,q}$ is a sequence such that
\begin{align*}
\|\lambda\|_{b_{p,q},\varrho}:=\sup_{\beta\in\mathbb{N}^N}2^{\varrho|\beta|}\|\lambda^\beta\|_{b_{p,q}}<\infty. 
\end{align*}
Moreover,
\begin{align}
\|f\|_{B_{p,q}^s(\mathbb{R}^N)}\sim \inf_{\eqref{G:dekompp}}\|\lambda\|_{b_{p,q},\varrho}\hspace{0.1em}, \label{G:ekiv55}
\end{align}
where the infimum is taken over all admissible representations \eqref{G:dekompp}. In addition, the right hand side of \eqref{G:ekiv55} is independent of the choice of $\psi$ and $\varrho>r$.
\end{theo}
This result can be found in \cite[Theorem 1.23, pp.35-36]{Mou} (see also \cite[Theorem 10, p.284]{Knopova}).
\begin{remark}
The counterpart of Remark \ref{optdecsub} for $B_{p,q}^{(s,\Psi)}(\R^N)$ remains valid, see \cite[Remark 1.26, p.48]{Mou}.
\end{remark}

\subsection{Related spaces and embeddings}\label{SU:spEM}
Let us now say a brief word about embeddings.
Given a locally integrable function $f$ in $\R^N$ and a set $B\subset\R^N$ having finite nonzero Lebesgue measure, we let
$$f_B:=\fint_{B}f(y)\,\mathrm{d}y=\frac{1}{\mathcal{L}_N(B)}\int_{B}f(y)\,\mathrm{d}y,$$
be the average of $f$ on $B$.

Moreover, we denote by $f^*:\R_+\to\R_+$ the \emph{decreasing rearrangement} of $f$, given by
$$ f^*(t):=\inf\big\{\lambda\geq0:\mu_f(\lambda)\leq t\big\}, $$
for all $t\geq0$, where
$$\mu_f(\lambda):=\mathcal{L}_N\left(\{x\in\R^N\!:|f(x)|>\lambda\}\right), $$
is the so-called \emph{distribution function} of $f$.
\begin{de}\label{spEM}
Let $s>0$ and $0<p<\infty$.
\begin{enumerate}
\item[$(\mathrm{i})$] The \emph{Zygmund-H\"older space} $C^s(\R^N)$ is the Besov space $B_{\infty,\infty}^s(\R^N)$.
\item[$(\mathrm{ii})$] The space of functions of \emph{bounded mean oscillation}, denoted by $\mathrm{BMO}(\R^N)$, consists of all locally integrable functions $f$ such that
\begin{align}
\|f\|_{\mathrm{BMO}(\R^N)}:=\sup_{B}\,\fint_{B}\left|f(x)-f_{B}\right|\mathrm{d}x<\infty, \label{DE:BMO}
\end{align}
where the supremum in \eqref{DE:BMO} is taken over all balls $B\subset\R^N$.
\item[$(\mathrm{iii})$] The \emph{weak} $L^p$\emph{-space}, denoted by $L^{p,\infty}(\R^N)$, consists of all measurable functions $f$ such that
\begin{align*}
\|f\|_{L^{p,\infty}(\R^N)}:=\sup_{t>0}\,t^{1/p}f^*(t)<\infty, 
\end{align*}
where $f^*$ is the decreasing rearrangement of $f$.
\end{enumerate}
\end{de}
Let us now state the following
\begin{theo}[Sobolev embedding theorem for $B_{p,q}^s$]
Let $0<p,q<\infty$ and $s>\sigma_p$.
\begin{enumerate}
\item[$(\mathrm{i})$] If $sp>N$, then $B_{p,q}^s(\R^N)\hookrightarrow C^{s-\frac{N}{p}}(\R^N)$. \vspace{3pt}
\item[$(\mathrm{ii})$] If $sp=N$, then $B_{p,q}^s(\R^N)\hookrightarrow \mathrm{BMO}(\R^N)$. 
\item[$(\mathrm{iii})$] If $sp<N$, then $B_{p,q}^s(\R^N)\hookrightarrow L^{\frac{Np}{N-sp},\infty}(\R^N)$. \vspace{3pt}
\end{enumerate}
In particular, $B_{p,q}^s(\R^N)\hookrightarrow A^{s,p}(\R^N)$ where $A^{s,p}(\R^N)$ is as in \eqref{Asp}.
\end{theo}
\begin{proof}
The cases (i), (ii) and (iii) are respectively covered by \cite[Formula (12), p.131]{Triebel}, \cite[Lemma 6.5]{MRS} and \cite[Th\'eor\`eme 8.1, p.301]{JPeetre}.
\end{proof}
\begin{remark}\label{WEIGHT:EMB}
Let us briefly mention that a corresponding result holds for the spaces $B_{p,q}^{(s,\Psi)}(\R^N)$. As already mentioned in the introduction, the space $B_{p,q}^{(s,\Psi)}(\R^N)$ is embedded in a generalized version of the H\"older space when $sp>N$. When $sp<N$, it is shown in \cite{Caetano} that $B_{p,q}^{(s,\Psi)}(\R^N)$ embeds in a weighted version of $L^{\frac{Np}{N-sp},\infty}(\R^N)$. Yet, when $sp=N$, the corresponding substitute for $\mathrm{BMO}$ does not seem to have been clearly identified nor considered in the literature, see however \cite{Caetano2,Gurka,Gurka2,MNS} where some partial results are given.
\end{remark}

\section{Preliminaries}\label{SE:SUITE}

In this section, we study the properties of some discrete sequences which will play an important role in the sequel. More precisely, we will be interested in the convergence of series of the type
$$ \sum_{j\geq0} 2^j|\lambda_{j,\lfloor 2^jx\rfloor}| \qquad{\mbox{ for }}~~x>0, $$
where $\lambda=(\lambda_{j,k})_{j,k\geq0}$ is an element of some Besov sequence space, say, $b_{1,q}$ with $q>1$.

\subsection{Some technical lemmata}
Let us start with a famous result due to Cauchy.

\begin{theo}[Cauchy's condensation test]
Let $\lambda\in\ell^1(\N)$ be a nonnegative, nonincreasing sequence. Then,
$$\sum_{j\geq1}\lambda_j\leq\,\sum_{j\geq0}2^j\lambda_{2^j}\leq 2\sum_{j\geq1}\lambda_j.$$
\end{theo}

\begin{remark}\label{nonincreazing}
The monotonicity assumption on $\lambda$ is central here. Indeed, there exist nonnegative sequences $\lambda\in\ell^1(\N)$ which are not nonincreasing and such that $\sum_{j\geq0}2^j\lambda_{2^j}=\infty$. Take for example:
\[
\lambda_j=\left\{
\begin{array}{c l}
1/k^2 & \text{if } j=2^k\hbox{ and }k\neq0, \vspace{3pt}\\
2^{-j} & \text{else}.
\end{array}
\right.
\]
Then, clearly, $\lambda\in\ell^1(\N)$. However, $2^j\lambda_{2^j}=\frac{2^j}{j^2}$ when $j\geq1$, so that $(2^j\lambda_{2^j})_{j\geq0}\notin\ell^1(\N)$.
\end{remark}

\noindent A simple consequence of Cauchy's condensation test is the following
\begin{lemma}\label{condens}
Let $\lambda\in\ell^1(\N)$ be a nonnegative, nonincreasing sequence. Then,
\begin{align*}
\sum_{j\geq0}2^j\lambda_{\lfloor 2^j x\rfloor}\leq \phi(x)\sum_{j\geq1}\lambda_j \qquad{\mbox{ for any }}~~x>0,
\end{align*}
where $\phi(x):=\frac{4}{|x|}\left(\mathds{1}_{[1,\infty)}(x)+(1-\log_2|x|)\mathds{1}_{(0,1)}(x)\right)$.
\end{lemma}
\begin{proof}
Let $k\in\mathbb{N}$ and $2^k\leq x\leq2^{k+1}$. Then, by Cauchy's condensation test
\begin{align}
\sum_{j\geq0}2^j\lambda_{\lfloor 2^j x\rfloor}&\leq\sum_{j\geq0}2^j\lambda_{2^{k+j}}=2^{-k}\sum_{j\geq k}2^j\lambda_{2^j}\leq 2^{-k}\sum_{j\geq 0}2^j\lambda_{2^j}\leq \frac{4}{x}\sum_{j\geq1}\lambda_j. \nonumber
\end{align}
In like manner, for $2^{-(k+1)}\leq x\leq 2^{-k}$, we have
\begin{align}
\sum_{j\geq0}2^j\lambda_{\lfloor 2^j x\rfloor}&\leq\sum_{j\geq0}2^j\lambda_{\lfloor 2^{j-k-1}\rfloor}=2^{k+1}\sum_{j\geq -k-1}2^j\lambda_{\lfloor2^j\rfloor}\leq 2^{k+1}(k+1)\sum_{j\geq0}2^j\lambda_{2^j}. \nonumber
\end{align}
Finally, invoking again Cauchy's condensation test, we have
\begin{align}
\sum_{j\geq0}2^j\lambda_{\lfloor 2^j x\rfloor}&\leq 2^{k+2}(k+1)\sum_{j\geq1}\lambda_j\leq \frac{4}{x}\left(1-\log_2(x)\right)\sum_{j\geq1}\lambda_j. \nonumber \qedhere
\end{align}
\end{proof}
In some sense, this "functional version" of Cauchy's condensation test may be generalized to sequences which are not necessarily nonincreasing. Indeed, one can show that
\[\mathcal{L}_1\bigg(\bigg\{x\in\mathbb{R}_+:\sum_{j\geq0}2^j|\lambda_{\lfloor 2^j x\rfloor}|=+\infty\bigg\}\bigg)=0,\]
whenever $\lambda\in\ell^1(\mathbb{N})$. This is due to the fact that $\ell^p$-spaces can be seen as "amalgams" of $L^p(1,2)$ and a weighted version of $\ell^p$. More precisely, we have
\begin{lemma}\label{ell1ekiv}
Let $0<p<\infty$ and let $\lambda\in\ell^p(\mathbb{N})$. Then
$$ \bigg(\sum_{j\geq1}|\lambda_j|^p\bigg)^{1/p}=\bigg(\int_{[1,2]}\,\sum_{j\geq0}2^j|\lambda_{\lfloor2^j x\rfloor}|^p\,\mathrm{d}x\bigg)^{1/p}.$$
\end{lemma}
\begin{proof}
It suffices to assume $p=1$ and that $\lambda$ is nonnegative. Then,
\begin{align*}
\frac{1}{2^k}\sum_{2^k\leq j<2^{k+1}}\lambda_j=\fint_{[2^k,2^{k+1}]}\lambda_{\lfloor x\rfloor}\,\mathrm{d}x=\frac{1}{2^k}\int_{[1,2]}\lambda_{\lfloor 2^ky\rfloor}2^k\,\mathrm{d}y=\int_{[1,2]}\lambda_{\lfloor2^k y\rfloor}\,\mathrm{d}y,
\end{align*}
which yields
\begin{align}
\sum_{j\in\mathbb{N}^\ast}\lambda_j&=\sum_{k\in\mathbb{N}}\sum_{2^k\leq j<2^{k+1}}\lambda_j= \sum_{k\in\mathbb{N}}2^k\int_{[1,2]}\lambda_{\lfloor2^k x\rfloor}\,\mathrm{d}x=\int_{[1,2]}\bigg(\sum_{k\in\mathbb{N}}\lambda_{\lfloor2^k x\rfloor}2^k\bigg)\mathrm{d}x. \nonumber \qedhere
\end{align}
\end{proof}
We now establish a technical inequality which we will be needed in the sequel.
\begin{lemma}\label{LE:tech}
Let $N\geq1$ and $0<p<\infty$. Let $\lambda=(\lambda_{j,k}^\beta)_{(j,\beta,k)\in\N\times\N^N\times\N^N}$ be a sequence such that the partial sequences $(\lambda_{j,k}^\beta)_{k\in\N^N}$ belong to $\ell^p(\N^N)$ for all $(j,\beta)\in\N\times\N^N$.
Then, for any positive $(\alpha_j)_{j\geq0}\in\ell^1(\N)$ there exists $C=C(\lambda,\alpha,N,d)>0$ such that for any $(j,\beta)\in\N\times\N^N$,
\begin{align*}
2^{j N}|\lambda_{j,\lfloor 2^j x\rfloor}^\beta|^p\leq C\,\frac{\max\{1,|\beta|^{N+1}\}}{\alpha_j}\sum_{k\in\N^N}|\lambda_{j,k}^\beta|^p, 
\end{align*}
holds for a.e. $x=(x_{1},...,x_N)\in[1,2]^{N}$ where
$$ \lfloor 2^j x\rfloor = (\lfloor 2^j x_{1}\rfloor, ... \,,\lfloor 2^j x_N\rfloor)\in\N^{N}. $$
\end{lemma}
\begin{proof}
For the sake of convenience, we use the following notations
\begin{align}
U_{j,\beta}(x):=2^{j N}|\lambda_{j,\lfloor 2^j x\rfloor}^\beta|^p\,\,~~~~~~~~\,\,\text{ and }\,\,~~~~\,\,~~~~U^{j,\beta}:=\sum_{k\in\N^N}|\lambda_{j,k}^\beta|^p. \nonumber
\end{align}
We have to prove that
$$ U_{j,\beta}(x)\leq C\frac{\max\{1,|\beta|^{N+1}\}}{\alpha_j}\,U^{j,\beta}, $$
for a.e. $x\in[1,2]^{N}$ and any $(j,\beta)\in\N\times\N^N$. By iterated applications of Lemma \ref{ell1ekiv}, we have
\begin{align}
\int_{[1,2]^{N}}U_{j,\beta}(x)\,\mathrm{d}x\leq U^{j,\beta}. \label{LP:condens}
\end{align}
Now, define
$$ \Gamma_{j,\beta}:=\left\{x\in[1,2]^{N}:U_{j,\beta}(x)\geq \frac{\max\{1,|\beta|^{N+1}\}}{\alpha_j}\,U^{j,\beta}\right\}. $$
Then, applying Markov's inequality and using \eqref{LP:condens}, we have
$$ \mathcal{L}_{N}(\Gamma_{j,\beta})\leq \frac{\alpha_j}{\max\{1,|\beta|^{N+1}\}} \qquad{\mbox{ for any }}~~(j,\beta)\in\N\times\N^{N}. $$
In turn, this gives
$$ \sum_{\beta\in\N^N}\sum_{j\geq0}\mathcal{L}_{N}(\Gamma_{j,\beta})<\infty. $$
Therefore, we can apply the Borel-Cantelli lemma and deduce that there exists $j_0,\beta_0\geq0$ such that
$$ U_{j,\beta}(x)\leq \frac{\max\{1,|\beta|^{N+1}\}}{\alpha_j}\,U^{j,\beta}, $$
for any $j>j_0$ and/or $|\beta|>\beta_0$ and a.e. $x\in[1,2]^{N}$. On the other hand, for any $j\leq j_0$ and $|\beta|\leq\beta_0$ we have
$$ U_{j,\beta}(x)\leq 2^{j_0N}\max\{1,|\beta|^{N+1}\}\frac{\max_{0\leq j\leq j_0}\alpha_j}{\alpha_j}\,U^{j,\beta}. $$
This completes the proof.
\end{proof}

\subsection{Some useful sequences}
We now construct some key sequences which will be at the crux of the proofs of Theorems \ref{THEOREM}, \ref{TH:BMO} and \ref{TH:WEIGHT33}.
\begin{lemma}\label{LE:CTR}
There exists a sequence $(\zeta_k)_{k\geq0}\subset\R_+$ satisfying
\begin{align}
\sup_{j\geq0}\bigg(\frac{1}{2^j}\sum_{2^j\leq k<2^{j+1}}\zeta_k\bigg)\leq 1, \label{LE:HYPO}
\end{align}
and such that
\begin{align}
\sup_{j\geq 0}\,\zeta_{\lfloor 2^jx\rfloor}=\infty \qquad{\mbox{ for all }}~~x\in [1,2). \label{LE:CONCL}
\end{align}
\end{lemma}
\begin{proof}
Let us first construct an auxiliary sequence satisfying \eqref{LE:HYPO}.

Let $(\lambda_k)_{k\geq0}$ be a sequence such that $\lambda_0=\lambda_1=0$ and such that, for any $j\geq 1$, the $\lfloor \frac{2^j}{j}\rfloor$ first terms of the sequence $(\lambda_k)_{k\geq0}$ on the discrete interval $[\![2^j,2^{j+1}-1]\!]$ have value $j$ and the remaining terms are all equal to zero. Then, for any $j\geq1$, we have
$$ \frac{1}{2^j}\sum_{2^j\leq k<2^{j+1}}\lambda_k=\frac{j+\cdots+j+0+\cdots+0}{2^j}=\frac{j\lfloor \frac{2^j}{j}\rfloor}{2^j}\leq 1. $$
For the sake of convenience, we set
$$T_j:=[\![2^j,2^{j+1}-1]\!] \qquad{\mbox{ for any }}~~j\geq0. $$
We will construct a sequence $(\zeta_k)_{k\geq0}$ satisfying both \eqref{LE:HYPO} and \eqref{LE:CONCL} by rearranging the terms of $(\lambda_k)_{k\geq0}$. To this end, we follow the following procedure.
\begin{figure}
\centering
\includegraphics[scale=0.7]{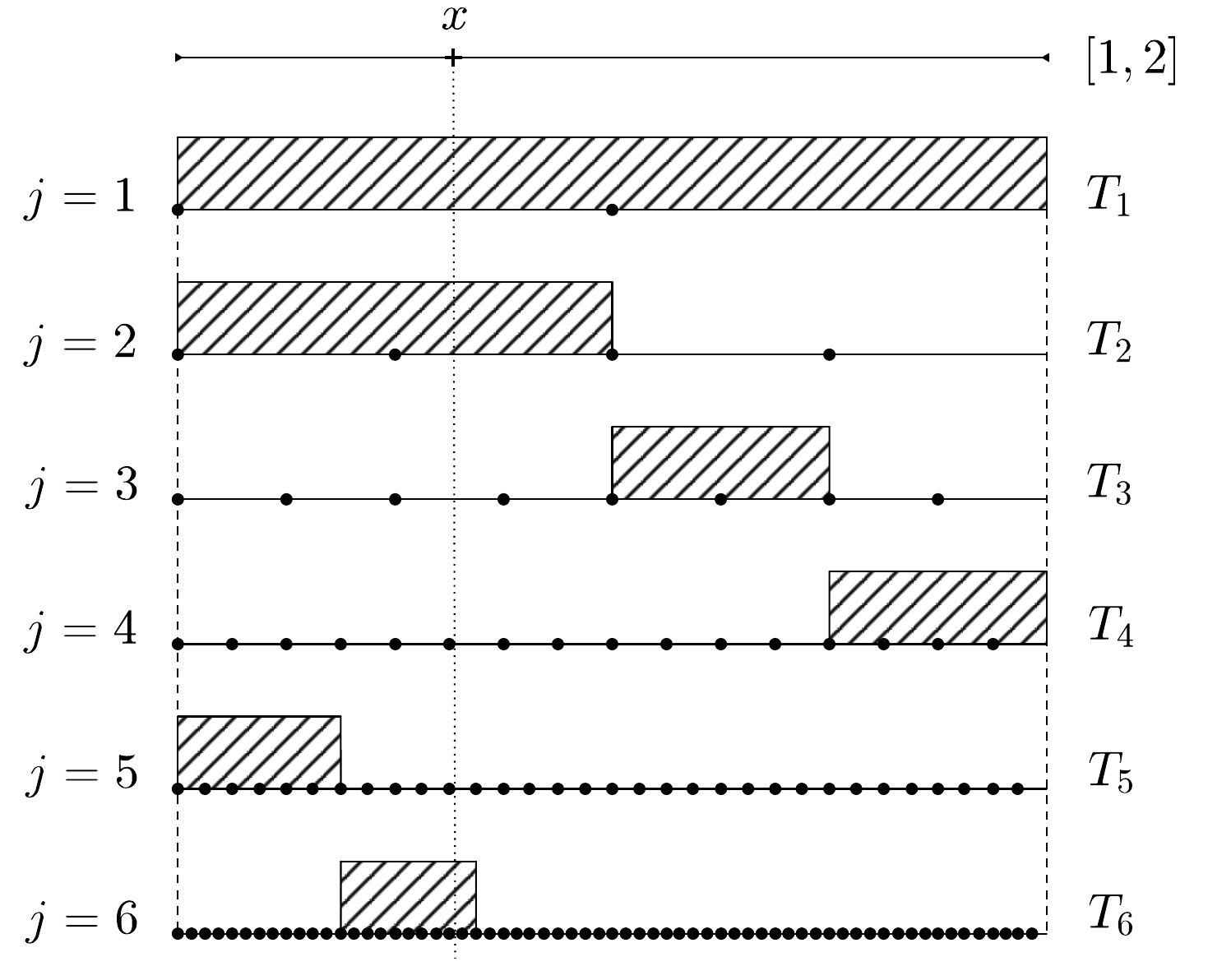} \\
\caption{Construction of the first terms of $(\zeta_k)_{k\geq0}$.} 
The hatched zone corresponds to the values of $x$ for which $\zeta_{\lfloor2^jx\rfloor}$ \\
takes nonzero values.
\label{FIG:SEQ}
\end{figure}
For $k\in[\![0,2^{3}-1]\!]$ we impose $\zeta_k=\lambda_k$. For $j=3$, we shift the values of $(\lambda_k)_{k\geq0}$ on $T_3$ in such a way that the smallest $x\in[1,2)$ such that $\zeta_{\lfloor 2^3x\rfloor}$ is nonzero coincides with the limit superior of the set of all $z\in[1,2)$ such that $\zeta_{\lfloor2^2z\rfloor}$ is nonzero. For $j=4$, we shift the values of $(\lambda_k)_{k\geq0}$ on $T_4$ in such a way that the smallest $x\in[1,2]$ such that $\zeta_{\lfloor2^4x\rfloor}$ is nonzero coincides with the limit superior of the set of all $z\in[1,2)$ such that $\zeta_{\lfloor2^3z\rfloor}$ is nonzero, and so on. When the range of nonzero terms has reached the last term on $T_j$ for some $j\geq1$, we start again from $T_{j+1}$ and set $\zeta_k=\lambda_k$ on $T_{j+1}$, and we repeat the above procedure. See Figure \ref{FIG:SEQ} for a visual illustration.

If, for some $j\geq0$, it happens that the above shifting of the $\lambda_k$'s on $T_j$ exceeds $T_j$, then we shift the $\lambda_k$'s on $T_j$  in such a way that the limit superior of the set of all $x\in[1,2]$ for which $\zeta_{\lfloor 2^jx\rfloor}$ is nonzero coincides with $x=2$.

Note that this procedure is well-defined because the proportion of nonzero terms on each $T_j$ is $2^{-j}\lfloor\frac{2^j}{j}\rfloor$ which has a divergent series thus allowing us to fill as much "space" as needed.

Then, by construction, for any $x\in[1,2)$ there are infinitely many values of $j\geq0$ such that $\zeta_{\lfloor2^jx\rfloor}=j$. Consequently, \eqref{LE:CONCL} holds. Moreover, \eqref{LE:HYPO} is trivially satisfied.

This completes the proof.
\end{proof}
As an immediate corollary, we have
\begin{cor}\label{COR:CTR2}
Let $0<p<q\leq\infty$. Then, there exists a sequence $(\lambda_{j,k})_{j,k\geq0}\subset\R_+$ satisfying
\begin{align*}
\bigg(\sum_{j\geq0}\bigg(\sum_{k\geq0}\lambda_{j,k}^p\bigg)^{q/p}\bigg)^{1/q}<\infty, 
\end{align*}
(modification if $q=\infty$) and such that
\begin{align*}
\sup_{j\geq 0}\,2^{j/p}\lambda_{j,\lfloor 2^jx\rfloor}=\infty \qquad{\mbox{ for all }}~~x\in [1,2). 
\end{align*}
\end{cor}
\begin{proof}
When $q=\infty$, it suffices to set
\begin{align}
\lambda_{j,k}=\left\{
\begin{array}{l l}
2^{-j/p}\zeta_k^{1/p} & \text{if }~2^j\leq k<2^{j+1}, \vspace{3pt}\\
0 & \text{otherwise},
\end{array}
\right. \label{casinfiniBZZ}
\end{align}
where $(\zeta_k)_{k\geq0}$ is the sequence constructed at Lemma \ref{LE:CTR}.

When $q<\infty$, we simply replace $(\zeta_k)_{k\geq0}$ in \eqref{casinfiniBZZ} by $(\xi_k)_{k\geq0}$ where
$$ \xi_k=j^{-\sqrt{\frac{p}{q}}} \zeta_k \qquad{\mbox{ for any }}~~k\in[\![2^j,2^{j+1}-1]\!]~~\text{ with }~~j\geq1, $$
and $\xi_0=\xi_1=0$. Then, we obtain
\begin{align}
\sum_{j\geq1}\bigg(\frac{1}{2^j}\sum_{2^j\leq k<2^{j+1}}\xi_k\bigg)^{q/p}=\sum_{j\geq1}\bigg(j^{-\sqrt{\frac{p}{q}}} \cdot \frac{1}{2^j}\sum_{2^j\leq k<2^{j+1}}\zeta_k\bigg)^{q/p}\leq\sum_{j\geq1}j^{-\sqrt{\frac{q}{p}}}<\infty. \nonumber
\end{align}
Moreover, by construction of $(\zeta_k)_{k\geq0}$, for any $x\in [1,2)$, there is a countably infinite set $J_x\subset\N$ such that $\zeta_{\lfloor2^jx\rfloor}=j$ for any $j\in J_x$. In particular,
$$ \xi_{\lfloor2^jx\rfloor}=j^{\alpha} \qquad{\mbox{ for any }}~~j\in J_x~~\text{ and }~~x\in[1,2), $$
where $\alpha=1-\sqrt{p/q}>0$. Thus,
$$ \sup_{j\geq0}\,2^j\lambda_{j,\lfloor2^jx\rfloor}^p=\sup_{j\geq0}\,\xi_{\lfloor2^jx\rfloor}\geq\sup_{j\in J_x}\,j^{\alpha}=\infty \qquad{\mbox{ for any }}~~x\in[1,2), $$
which is what we had to show.
\end{proof}
We conclude this section by a weighted version of Corollary \ref{COR:CTR2}.
\begin{lemma}\label{LE:CTR:WEIGHT3}
Let $0<p<q\leq\infty$. Let $\Psi$ be an admissible function that does not satisfy \eqref{H:WEIGHT}. If $q<\infty$ and $\Psi$ is increasing assume, in addition, that
$$ \chi=\frac{qp}{q-p}<\frac{1}{c_\infty}, $$
where $c_\infty$ is as in Theorem \ref{TH:WEIGHT33}.
Then, there exists a sequence $(\lambda_{j,k})_{j,k\geq0}\subset\R_+$ such that
\begin{align}
\bigg(\sum_{j\geq0}\bigg(\sum_{k\geq0}\lambda_{j,k}^p\bigg)^{q/p}\bigg)^{1/q}<\infty, \label{C2W}
\end{align}
(modification if $q=\infty$) and
\begin{align}
\bigg(\sum_{j\geq0}\,2^{j\frac{q}{p}}\lambda_{j,\lfloor2^jx\rfloor}^q\Psi(2^{-j})^q\bigg)^{1/q}=\infty \qquad{\mbox{ for all }}~~x\in[1,2), \label{C3W}
\end{align}
(modification if $q=\infty$).
\end{lemma}
\begin{proof}
The proof is essentially the same as in the unweighted case with minor changes that we shall now detail.

Let us begin with the case $q=\infty$. Let $\beta_j:=\Psi(2^{-j})^p$. Since $\beta_j>0$ and $(\beta_j)_{j\geq0}\notin\ell^1(\N)$ we may find another positive sequence $(\gamma_j)_{j\geq0}$ which has a divergent series and such that
$$ \frac{\beta_j}{\gamma_j}\to\infty \qquad{\mbox{ as }}~~j\to\infty, $$
i.e. $(\gamma_j)_{j\geq0}$ diverges slower than $(\beta_j)_{j\geq0}$. Take, for example
$$ \gamma_j=\frac{\beta_j}{\sum_{k=0}^j\beta_k}, $$
see e.g. \cite{Ash}. Note that $0<\gamma_j\leq1$ for all $j\geq0$. Let $(\varrho_k)_{k\geq0}$ be a sequence such that $\varrho_0=\varrho_1=0$ and such that, for any $j\geq 1$, the $\lfloor 2^j\gamma_j\rfloor$ first terms of the sequence $(\varrho_k)_{k\geq0}$ on the discrete interval $T_j:=[\![2^j,2^{j+1}-1]\!]$ have value $\frac{1}{\gamma_j}$ and the remaining terms are all equal to zero. Then, for any $j\geq1$, we have
$$ \frac{1}{2^j}\sum_{2^j\leq k<2^{j+1}}\varrho_k=\frac{\frac{1}{\gamma_j}+\cdots+\frac{1}{\gamma_j}+0+\cdots+0}{2^j}=\frac{\frac{1}{\gamma_j}\lfloor 2^j\gamma_j\rfloor}{2^j}\leq 1. $$
Now, since the proportion of nonzero terms on each $T_j$ is $2^{-j}\lfloor 2^j\gamma_j\rfloor$ which has a divergent series, we may apply to $(\varrho_j)_{j\geq0}$ the same rearrangement as in the proof of Lemma \ref{LE:CTR}. That is, we can construct a sequence $(\varrho_j^*)_{j\geq0}$ such that
\begin{align*}
\frac{1}{2^j}\sum_{2^j\leq k<2^{j+1}}\varrho_k^*\leq1 \qquad{\mbox{ for all }}~~j\geq0, 
\end{align*}
and for any $x\in[1,2)$ there is a countably infinite set $J_x\subset\N$ such that
\begin{align*}
\beta_j\varrho_{\lfloor2^jx\rfloor}^*= \frac{\beta_j}{\gamma_j}=\sum_{k=0}^j\beta_k \qquad{\mbox{ for all }}~~j\in J_x, 
\end{align*}
i.e. we have
$$ \sup_{j\geq0}\,\beta_j\varrho_{\lfloor2^jx\rfloor}^*\geq \lim_{j\to\infty \atop j\in J_x}\,\sum_{k=0}^j\beta_k=\infty. $$
Therefore, letting
\begin{align}
\lambda_{j,k}=\left\{
\begin{array}{l l}
2^{-j/p}(\varrho_k^*)^{1/p} & \text{if }~2^j\leq k<2^{j+1}, \vspace{3pt}\\
0 & \text{otherwise},
\end{array}
\right. \nonumber
\end{align}
we obtain a sequence satisfying both \eqref{C2W} and \eqref{C3W}.

Let us now prove the lemma when $q<\infty$. Notice that if $\Psi$ is either constant or decreasing there is nothing to prove since the result is a consequence of Corollary \ref{COR:CTR2}. Hence, we may assume that $\Psi$ is increasing. By our assumptions, we have
$$ \frac{1}{p}\log_2\frac{\beta_j}{\beta_{2j}}\leq c_\infty<\frac{1}{\chi}, $$
which implies that
\begin{align}
\beta_j\leq 2^{c_\infty p}\beta_{2j}\leq\cdots\leq2^{kc_\infty p}\beta_{2^kj} \qquad{\mbox{ for any }}~~k\in\N. \label{prop:betajcroiss}
\end{align}
By Cauchy's condensation test, we have
$$ 2\sum_{j\geq0}\beta_j^{\frac{q}{q-p}}\geq\sum_{j\geq0}2^j\beta_{2^j}^{\frac{q}{q-p}}\geq \beta_1^{\frac{q}{q-p}}\sum_{j\geq0}2^{j(1-c_\infty\chi)}=\infty. $$
Thus, we may infer as above that the following positive sequence has divergent series:
$$ \tilde{\gamma}_j=\frac{\beta_j^{\frac{q}{q-p}}}{\sum_{k=0}^j\beta_k^{\frac{q}{q-p}}}. $$
Notice that $\tilde{\gamma}_j\leq1/(j+1)$.
Now define $(\tau_j)_{j\geq0}$ by $\tau_j:=\tilde{\gamma}_j/\beta_j$.
Since $2^{-c_\infty p}\beta_1\leq j^{c_\infty p}\beta_j$ for any $j\geq1$ (by \eqref{prop:betajcroiss} and the monotonicity of $\beta_j$), our assumptions on $\chi$ and $c_\infty$ then imply
$$ \sum_{j\geq0}\tau_j^{q/p}\leq \tau_0^{q/p}+\sum_{j\geq1}\frac{\beta_j^{-q/p}}{(j+1)^{q/p}}\leq\tau_0^{q/p}+\sum_{j\geq1}\frac{2^{c_\infty q}\beta_1^{-q/p}}{j^{q(1/p-c_\infty)}}<\infty. $$
The conclusion now follows by letting $\tilde{\lambda}_{j,k}:=\tau_j^{1/p}\lambda_{j,k}$ where $\lambda_{j,k}$ is the sequence constructed above with $\tilde{\gamma}_j$ instead of $\gamma_j$. Indeed, we have
$$ \sum_{j\geq0}\bigg(\sum_{k\geq0}\tilde{\lambda}_{j,k}^p\bigg)^{q/p}\leq \sum_{j\geq0}\tau_j^{q/p}<\infty, $$
and, for each $x\in[1,2)$, there is a countably infinite set $\tilde{J}_x\subset\N$ such that
$$ 2^{j/p}\tilde{\lambda}_{j,\lfloor2^jx\rfloor}\beta_j^{1/p}=2^{j/p}\,\tau_j^{1/p}\left(\frac{\beta_j}{\tilde{\gamma}_j}\right)^{1/p}2^{-j/p}=1 \qquad{\mbox{ for any }}~~j\in \tilde{J}_x. $$
Therefore, $(2^{j/p}\tilde{\lambda}_{j,\lfloor2^jx\rfloor}\beta_j^{1/p})_{j\geq0}\notin\ell^{q}(\N)$.
This completes the proof.
\end{proof}

\section{General estimates}\label{SE:ESTIMATE}

Throughout this section we will write $x\in\R^N$ as $x=(x_1,...,x_N)=(x',x'')$ with $x'\in\R^d$, $x''\in\R^{N-d}$ and, similarly, $m=(m',m'')\in\Z^N$ and $\beta=(\beta',\beta'')\in\N^N$. Also, we set
$$ \mathcal{D}:=\{0,1\}^{N-d}. $$
Let $\psi\in C_0^\infty(\mathbb{R}^N,[0,1])$ be such that $\mathrm{supp}(\psi)\subset B_1$ and that
\[2^{-\nu(s-\frac{N}{p})}\psi^\beta(2^\nu x-m),\]
are $(s,p)$-$\beta$-quarks. Also, we assume that $\psi$ has the product structure
\begin{align}
\psi(x_1,...,x_N)=\psi(x_1)\,...\,\psi(x_N). \label{prodstruc}
\end{align}
Let $\varrho>0$ and $f\in B_{p,q}^s(\mathbb{R}^N)$. Then, by Theorem \ref{BspqQuark}, there are coefficients $\lambda_{\nu,m}^\beta$ such that
\begin{align}
&f(x)=\sum_{\beta\in\mathbb{N}^N}\sum_{\nu=0}^{\infty}\sum_{m\in\mathbb{Z}^N}\lambda_{\nu,m}^\beta2^{-\nu(s-\frac{N}{p})}\psi^\beta(2^\nu x-m). \label{rep3000}
\end{align}
We can further assume that
\begin{align}
\|f\|_{B_{p,q}^s(\mathbb{R}^N)}&\sim \sup_{\beta\in\mathbb{N}^N}2^{\varrho|\beta|}\bigg(\sum_{\nu\geq0}\bigg(\sum_{m\in\mathbb{Z}^N}|\lambda_{\nu,m}^\beta|^p \bigg)^{q/p}\bigg)^{1/q}, \label{2norm}
\end{align}
i.e. that $\lambda_{\nu,m}^\beta=\lambda_{\nu,m}^\beta(f)$ is an \emph{optimal subatomic decomposition} of $f$.
Note, however, that the optimality of the decomposition $\lambda_{\nu,m}^\beta(f)$ depends on the choice of $\varrho>0$ (this can be seen from \cite[Corollary 2.12, p.23]{Trieb}). Of course, by Theorem \ref{BspqQuark}, we still have
\begin{align*}
\|f\|_{B_{p,q}^s(\mathbb{R}^N)}&\lesssim \sup_{\beta\in\mathbb{N}^N}2^{\varrho'|\beta|}\bigg(\sum_{\nu\geq0}\bigg(\sum_{m\in\mathbb{Z}^N}|\lambda_{\nu,m}^\beta|^p \bigg)^{q/p}\bigg)^{1/q}, 
\end{align*}
for any positive $\varrho'\ne\varrho$. Using \eqref{prodstruc} and \eqref{2norm}, we can decompose $f(\cdot,x'')$ as
\begin{align}
f(x',x'')&=\sum_{\nu\geq0}\sum_{\beta'\in\mathbb{N}^d}\sum_{m'\in\mathbb{Z}^d}b_{\nu,m'}^{\beta'}(\lambda,x'')2^{-\nu(s-\frac{d}{p})}\psi^{\beta'}(2^\nu x'-m'), \nonumber
\end{align}
where we have set
\begin{align}
b_{\nu,m'}^{\beta'}(\lambda,x''):=2^{\nu\frac{N-d}{p}}\sum_{\beta''\in\mathbb{N}^{N-d}}\sum_{m''\in\mathbb{Z}^{N-d}}\lambda_{\nu,m}^\beta\psi^{\beta''}(2^\nu x''-m''). \label{DEFb}
\end{align}
Then, defining
\begin{align}
J_{p,q}^{\varrho}(\lambda,x''):=\sup_{\beta'\in\mathbb{N}^d}2^{\varrho|\beta'|}\bigg(\sum_{\nu\geq0}\bigg(\sum_{m'\in\mathbb{Z}^d}|b_{\nu,m'}^{\beta'}(\lambda,x'')|^p\bigg)^{q/p}\bigg)^{1/q},
\label{DEFj}
\end{align}
we obtain
\[\|f(\cdot,x'')\|_{B_{p,q}^s(\mathbb{R}^d)}\lesssim J_{p,q}^{\varrho}(\lambda,x'').\]
In fact, we also have
\begin{align}
\|f(\cdot,x'')\|_{B_{p,q}^s(\mathbb{R}^d)}\lesssim J_{p,q}^{\varrho'}(\lambda,x''), \label{JJJ}
\end{align}
for any $\varrho'>0$. For the sake of convenience, we introduce some further notations. Given any $\delta\in\mathcal{D}$, we set
\begin{align}
b_{\nu,m'}^{\beta',\delta}(\lambda,x'')&:= 2^{\nu\frac{N-d}{p}}\sum_{\beta''\in\mathbb{N}^{N-d}}|\lambda_{\nu,m',\lfloor2^\nu x_{d+1}\rfloor+\delta_{d+1},...,\lfloor2^\nu x_N\rfloor+\delta_N}^\beta|, \label{newB} \\
J_{p,q}^{\varrho,\delta}(\lambda,x')&:=\sup_{\beta'\in\mathbb{N}^d}2^{\varrho|\beta'|}\bigg(\sum_{\nu\geq0}\bigg(\sum_{m'\in\mathbb{Z}^d}|b_{\nu,m'}^{\beta',\delta}(\lambda,x'')|^p\bigg)^{q/p}\bigg)^{1/q}. \label{newJ}
\end{align}
Notice that since $\mathrm{supp}(\psi^\beta)\subset B_1$, we have
\begin{align}
\psi^{\beta''}(2^\nu x''-m'')\ne0~~\Longrightarrow~~m_i\in\{\lfloor 2^\nu x_i\rfloor,\lfloor 2^\nu x_i\rfloor+1\}~~\text{for all}~~i\in[\![d+1,N]\!]. \nonumber
\end{align}
And so, using \eqref{DEFb} and \eqref{DEFj}, we can derive the following bounds
\begin{align*}
b_{\nu,m'}^{\beta'}(\lambda,x'')&\leq \sum_{\delta\in\mathcal{D}}b_{\nu,m'}^{\beta',\delta}(\lambda,x''),
\end{align*}
and
\begin{align}
J_{p,q}^{\varrho}(\lambda,x'')&\leq c\sum_{\delta\in\mathcal{D}}J_{p,q}^{\varrho,\delta}(\lambda,x''), \label{JJ}
\end{align}
for some $c>0$ depending only on $\#\mathcal{D}$, $p$ and $q$.

As a consequence of \eqref{JJJ} and \eqref{JJ}, to estimate $\|f(\cdot,x'')\|_{B_{p,q}^s(\R^d)}$ from above one only needs to estimate the terms \eqref{newJ} from above, for each $\delta\in\mathcal{D}$.

Within these notations, we have the following

\begin{lemma}\label{dommm}
Let $N\geq2$, $0<p,q\leq\infty$, $\delta\in\mathcal{D}$ and $0<\varrho'<\varrho_0$. Then, with the notations above
\begin{align*}
J_{p,q}^{\varrho',\delta}(\lambda,x'')\lesssim \sup_{\beta\in\mathbb{N}^N}2^{\varrho_0|\beta|}\bigg(\sum_{\nu\geq0}\bigg(\sum_{m'\in\mathbb{Z}^d}|\lambda_{\nu,m',\lfloor2^\nu x''\rfloor+\delta}^\beta|^p 2^{\nu(N-d)}\bigg)^{q/p}\bigg)^{1/q}, 
\end{align*}
for a.e. $x''\in\R^{N-d}$ (modification if $p=\infty$ and/or $q=\infty$) where
$$ \lfloor2^\nu x''\rfloor+\delta=(\lfloor2^\nu x_{d+1}\rfloor+\delta_{d+1},...\,,\lfloor2^\nu x_N\rfloor+\delta_N)\in\Z^{N-d}.$$
\end{lemma}
\begin{proof}
Suppose first that $p,q<\infty$. For simplicity, we will write
\begin{align*}
m_{\nu,\delta}'':=\lfloor2^\nu x''\rfloor+\delta.
\end{align*}
Using \eqref{newB} and \eqref{newJ} we get
\begin{align}
J_{p,q}^{\varrho',\delta}(\lambda,x'')&\lesssim \sup_{\beta'\in\mathbb{N}^d}2^{\varrho'|\beta'|}\bigg(\sum_{\nu\geq0}\bigg(\sum_{m'\in\mathbb{Z}^d}\bigg(\sum_{\beta''\in\mathbb{N}^{N-d}}|\lambda_{\nu,m',m_{\nu,\delta}''}^\beta|\bigg)^p 2^{\nu(N-d)}\bigg)^{q/p}\bigg)^{1/q}. \nonumber
\end{align}
Write $\lambda_{\nu,m}^\beta=2^{-\varrho_0|\beta|}\Lambda_{\nu,m}^\beta$ and let $a:=\varrho_0-\varrho'$. Then,
$$ 2^{\varrho'|\beta'|}|\lambda_{\nu,m}^\beta|=2^{\varrho'|\beta'|-\varrho_0|\beta|}|\Lambda_{\nu,m}^\beta|\leq2^{-a|\beta|}|\Lambda_{\nu,m}^\beta|\leq 2^{-a|\beta''|}|\Lambda_{\nu,m}^\beta|. $$
Hence, by H\"older's inequality we have
\begin{align}
J_{p,q}^{\varrho',\delta}(\lambda,x'')&\lesssim
\sup_{\beta'\in\mathbb{N}^d}\bigg(\sum_{\nu\geq0}\bigg(\sum_{m'\in\mathbb{Z}^d}\bigg(\sum_{\beta''\in\mathbb{N}^{N-d}}2^{-a|\beta''|}|\Lambda_{\nu,m',m_{\nu,\delta}''}^\beta|\bigg)^p 2^{\nu(N-d)}\bigg)^{q/p}\bigg)^{1/q} \nonumber \\
&\leq K_{a/2}\sup_{\beta'\in\mathbb{N}^d}\bigg(\sum_{\nu\geq0}\bigg(\sum_{m'\in\mathbb{Z}^d}\sup_{\beta''\in\mathbb{N}^{N-d}}2^{-p\frac{a}{2}|\beta''|}|\Lambda_{\nu,m',m_{\nu,\delta}''}^\beta|^p 2^{\nu(N-d)}\bigg)^{q/p}\bigg)^{1/q}, \nonumber
\end{align}
where we have used the notation
\begin{align*}
K_\alpha=\sum_{\beta''\in\mathbb{N}^{N-d}}2^{-\alpha|\beta''|} \qquad{\mbox{ for }}~~\alpha>0.
\end{align*}
Since the $\ell^p$ spaces are increasing with $p$, by successive applications of the H\"older inequality, we have
\begin{align}
J_{p,q}^{\varrho',\delta}&(\lambda,x'')\lesssim K_{a/2}\sup_{\beta'\in\mathbb{N}^d}\bigg(\sum_{\nu\geq0}\bigg(\sum_{\beta''\in\mathbb{N}^{N-d}}2^{-p\frac{a}{2}|\beta''|}\sum_{m'\in\mathbb{Z}^d}|\Lambda_{\nu,m',m_{\nu,\delta}''}^\beta|^p 2^{\nu(N-d)}\bigg)^{q/p}\bigg)^{1/q} \nonumber \\
&\leq K_{a/2}K_{p\frac{a}{4}}^{1/p}\sup_{\beta'\in\mathbb{N}^d}\bigg(\sum_{\nu\geq0}\bigg(\sup_{\beta''\in\mathbb{N}^{N-d}}2^{-p\frac{a}{4}|\beta''|}\sum_{m'\in\mathbb{Z}^d}|\Lambda_{\nu,m',m_{\nu,\delta}''}^\beta|^p 2^{\nu(N-d)}\bigg)^{q/p}\bigg)^{1/q} \nonumber \\
&= K_{a/2}K_{p\frac{a}{4}}^{1/p}\sup_{\beta'\in\mathbb{N}^d}\bigg(\sum_{\nu\geq0}\sup_{\beta''\in\mathbb{N}^{N-d}}2^{-q\frac{a}{4}|\beta''|}\bigg(\sum_{m'\in\mathbb{Z}^d}|\Lambda_{\nu,m',m_{\nu,\delta}''}^\beta|^p 2^{\nu(N-d)}\bigg)^{q/p}\bigg)^{1/q} \nonumber \\
&\leq K_{a/2}K_{p\frac{a}{4}}^{1/p}\sup_{\beta'\in\mathbb{N}^d}\bigg(\sum_{\nu\geq0}\sum_{\beta''\in\mathbb{N}^{N-d}}2^{-q\frac{a}{4}|\beta''|}\bigg(\sum_{m'\in\mathbb{Z}^d}|\Lambda_{\nu,m',m_{\nu,\delta}''}^\beta|^p 2^{\nu(N-d)}\bigg)^{q/p}\bigg)^{1/q} \nonumber \\
&= K_{a/2}K_{p\frac{a}{4}}^{1/p}\sup_{\beta'\in\mathbb{N}^d}\bigg(\sum_{\beta''\in\mathbb{N}^{N-d}}2^{-q\frac{a}{4}|\beta''|}\sum_{\nu\geq0}\bigg(\sum_{m'\in\mathbb{Z}^d}|\Lambda_{\nu,m',m_{\nu,\delta}''}^\beta|^p 2^{\nu(N-d)}\bigg)^{q/p}\bigg)^{1/q} \nonumber \\
&\leq K_{a/2}K_{p\frac{a}{4}}^{1/p}K_{q\frac{a}{4}}^{1/q}\sup_{\beta\in\mathbb{N}^N}\bigg(\sum_{\nu\geq0}\bigg(\sum_{m'\in\mathbb{Z}^d}|\Lambda_{\nu,m',m_{\nu,\delta}''}^\beta|^p 2^{\nu(N-d)}\bigg)^{q/p}\bigg)^{1/q}. \nonumber
\end{align}
Letting $K_{a,p,q}:=K_{a/2}K_{p\frac{a}{4}}^{1/p}K_{q\frac{a}{4}}^{1/q}$ and recalling $\lambda_{\nu,m}^\beta=2^{-\varrho_0|\beta|}\Lambda_{\nu,m}^\beta$ we get
\begin{align}
J_{p,q}^{\varrho',\delta}(\lambda,x'')&\leq K_{a,p,q}\sup_{\beta\in\mathbb{N}^N}2^{\varrho_0|\beta|}\bigg(\sum_{\nu\geq0}\bigg(\sum_{m'\in\mathbb{Z}^d}|\lambda_{\nu,m',m_{\nu,\delta}''}^\beta|^p 2^{\nu(N-d)}\bigg)^{q/p}\bigg)^{1/q}, \nonumber
\end{align}
which is the desired estimate. The proof when $p=\infty$ and/or $q=\infty$ is similar but technically simpler.
\end{proof}
\begin{remark}\label{RKdommm}
Of course, when $p=\infty$, the term "$2^{\nu(N-d)}$" disappears (recall Remark \ref{infiniQuark}) so that, in this case, Fact \ref{FACT1} follows directly from the above lemma.
\end{remark}
\begin{remark}\label{R:dommm}
The same kind of estimate holds in the setting of Besov spaces of generalized smoothness. That is, given a function $f\in B_{p,q}^s(\R^N)$ decomposed as above by \eqref{rep3000}
with \eqref{prodstruc} and \eqref{2norm}, we can estimate the $B_{p,q}^{(s,\Psi)}(\R^d)$-quasi-norm of its restrictions to almost every hyperplanes $f(\cdot,x'')$ exactly in the same fashion. It suffices to replace the $(s,p)$-$\beta$-quarks $(\beta\mathrm{qu})_{\nu,m}$ in the decomposition of $f(\cdot,x'')$ by $\Psi(2^{-\nu})^{-1}(\beta\mathrm{qu})_{\nu,m}$ in order to get $(s,p,\Psi)$-$\beta$-quarks. From here, we can reproduce the same reasoning as in Lemma \ref{dommm} with $\Psi(2^{-\nu})\lambda_{\nu,m}^\beta$ instead of $\lambda_{\nu,m}^\beta$ and we obtain
$$ \|f(\cdot,x'')\|_{B_{p,q}^{(s,\Psi)}(\R^d)}\lesssim\sum_{\delta\in\mathcal{D}}\tilde{J}_{p,q}^{\varrho_0,\delta}(\lambda,x''), $$
with
$$ \tilde{J}_{p,q}^{\varrho_0,\delta}(\lambda,x''):=\sup_{\beta\in\mathbb{N}^N}2^{\varrho_0|\beta|}\bigg(\sum_{\nu\geq0}\bigg(\Psi(2^{-\nu})^p\sum_{m'\in\mathbb{Z}^d}|\lambda_{\nu,m',\lfloor2^\nu x''\rfloor+\delta}^\beta|^p 2^{\nu(N-d)}\bigg)^{q/p}\bigg)^{1/q}. $$
Similarly, given a function $f\in B_{p,q}^{(s,\Psi)}(\R^N)$, we can estimate the $B_{p,q}^{(s,\Psi)}(\R^d)$-quasi-norm of its restrictions $f(\cdot,x'')$ in the same spirit. This is done up to a slight modification in the discussion above. It suffices to multiply the $(s,p)$-$\beta$-quarks considered above by a factor of $\Psi(2^{-\nu})^{-1}$ and to take $\eta_{\nu,m}^\beta$, the optimal subatomic decomposition of $f\in B_{p,q}^{(s,\Psi)}(\R^N)$ with respect to these new quarks. Then, the $B_{p,q}^{(s,\Psi)}(\R^N)$-quasi-norm of $f$ and the $B_{p,q}^{(s,\Psi)}(\R^d)$-quasi-norm of its restrictions $f(\cdot,x'')$ satisfy the same relations as when $\Psi\equiv1$ with $\eta_{\nu,m}^\beta$ instead of $\lambda_{\nu,m}^\beta$. That is, we still have
$$ \|f\|_{B_{p,q}^{(s,\Psi)}(\R^N)}\sim \sup_{\beta\in\mathbb{N}^N}2^{\varrho|\beta|}\bigg(\sum_{\nu\geq0}\bigg(\sum_{m\in\mathbb{Z}^N}|\eta_{\nu,m}^\beta|^p \bigg)^{q/p}\bigg)^{1/q}, $$
and
$$ \|f(\cdot,x'')\|_{B_{p,q}^{(s,\Psi)}(\R^d)}\lesssim \sum_{\delta\in\mathcal{D}} J_{p,q}^{\varrho_0,\delta}(\eta,x''), $$
where $\varrho,\varrho_0>0$ and $J_{p,q}^{\varrho_0,\delta}(\eta,x'')$ is as in \eqref{newJ}.
\end{remark}

\section{The case $q\leq p$} \label{SE:FACILE}

This section is concerned with Fact \ref{FACT1} (Fact \ref{FACT2} being only a consequence of Theorem \ref{TH:WEIGHT}).
We will use subatomic decompositions together with the estimate given at Lemma \ref{dommm} to get the following generalization of Fact \ref{FACT1}.
\begin{prop}\label{caracBesov2}
Let $N\geq2$, $1\leq d<N$, $0<q\leq p\leq\infty$ and $s>\sigma_p$. Let $\Psi$ be an admissible function. Let $K\subset\R^{N-d}$ be a compact set and let $f\in B_{p,q}^{(s,\Psi)}(\R^N)$. Then,
\begin{align*}
\left(\int_{K}\|f(\cdot,x'')\|_{B_{p,q}^{(s,\Psi)}(\R^d)}^q\mathrm{d}x''\right)^{1/q}\leq C\|f\|_{B_{p,q}^{(s,\Psi)}(\R^N)},
\end{align*}
for some constant $C=C(K,N,d,p,q)>0$ (modification if $q=\infty$).
\end{prop}
\begin{proof}
Without loss of generality, we may consider the case $K=[1,2]^{N-d}$ only (the general case follows from standard scaling arguments). Also, we can suppose that $p<\infty$ since otherwise, when $p=\infty$, the desired result is a simple consequence of Lemma \ref{dommm} (recall Remark \ref{RKdommm}). Let us first prove Lemma \ref{caracBesov2} for $\Psi\equiv1$ (it will be clear at the end why this is enough to deduce the general case).

Let $f\in B_{p,q}^s(\R^N)$. Given the $(s,p)$-$\beta$-quarks $(\beta\mathrm{qu})_{\nu,m}$ and $\varrho>r$ defined at Section \ref{SE:ESTIMATE} we let $\lambda_{\nu,m}^\beta=\lambda_{\nu,m}^\beta(f)$ be the corresponding optimal subatomic decomposition. In particular
$$ f(x)=\sum_{\beta\in\N^N}\sum_{\nu=0}^\infty\sum_{m\in\Z^N}\lambda_{\nu,m}^\beta(\beta\mathrm{qu})_{\nu,m}(x), $$
with
$$ \|f\|_{B_{p,q}^s(\R^N)}\sim\sup_{\beta\in\N^N}2^{\varrho|\beta|}\bigg(\sum_{\nu\geq 0}\bigg(\sum_{m\in\Z^{N}}|\lambda_{\nu,m}^\beta|^p\bigg)^{q/p}\bigg)^{1/q}. $$
By the discussion in Section \ref{SE:ESTIMATE}, we have that
\begin{align}
\|f(\cdot,x'')\|_{B_{p,q}^s(\R^d)}\lesssim \sum_{\delta\in\mathcal{D}}J_{p,q}^{\varrho',\delta}(\lambda,x''), \label{majoration22}
\end{align}
for all $\varrho'\in(0,\varrho)$, where $J_{p,q}^{\varrho',\delta}(\lambda,x'')$ is given by \eqref{newJ}. Define
$$ \Lambda_{\nu,m''}^\beta:=\bigg(\sum_{m'\in\Z^d}|\lambda_{\nu,m',m''}^\beta|^p\bigg)^{1/p}. $$
In particular,
$$ \|f\|_{B_{p,q}^s(\R^N)}\sim\sup_{\beta\in\N^N}2^{\varrho|\beta|}\bigg(\sum_{\nu\geq 0}\bigg(\sum_{m''\in\Z^{N-d}}|\Lambda_{\nu,m''}^\beta|^p\bigg)^{q/p}\bigg)^{1/q}. $$
Then, the conclusion of Lemma \ref{dommm} rewrites
$$ J_{p,q}^{\varrho',\delta}(\lambda,x'')^q\lesssim \sup_{\beta\in\N^N}2^{\varrho_0 q|\beta|}\sum_{\nu\geq0}2^{\nu q\frac{N-d}{p}}|\Lambda_{\nu,\lfloor2^\nu x''\rfloor+\delta}^\beta|^q \qquad{\mbox{ for all }}~~~~\delta\in\mathcal{D}, $$
and some $\varrho_0\in(\varrho',\varrho)$.
Integration over $[1,2]^{N-d}$ yields
\begin{align}
I_\delta:=\int_{[1,2]^{N-d}}&J_{p,q}^{\varrho',\delta}(\lambda,x'')^q\mathrm{d}x''\lesssim \int_{[1,2]^{N-d}}\sup_{\beta\in\N^N}2^{\varrho_0 q|\beta|}\sum_{\nu\geq0}2^{\nu q\frac{N-d}{p}}|\Lambda_{\nu,\lfloor2^\nu x''\rfloor+\delta}^\beta|^q\mathrm{d}x'' \nonumber \\
&\leq \sum_{\beta\in\N^N}2^{\varrho_0 q|\beta|}\sum_{\nu\geq0}\int_{[1,2]^{N-d}}2^{\nu q\frac{N-d}{p}}|\Lambda_{\nu,\lfloor2^\nu x''\rfloor+\delta}^\beta|^q\mathrm{d}x''. \nonumber
\end{align}
Now, we observe that
\begin{align}
2^{\nu q\frac{N-d}{p}}|\Lambda_{\nu,\lfloor2^\nu x''\rfloor+\delta}^\beta|^q&\leq \bigg(\sum_{k\in\N^{N-d}}|\Lambda_{\nu,\lfloor2^{k_{d+1}} x_{d+1}\rfloor+\delta_{d+1},\cdots,\lfloor2^{k_{N}} x_{N}\rfloor+\delta_{N}}^\beta|^p2^{k_{d+1}+\cdots+k_N}\bigg)^{q/p}. \nonumber
\end{align}
Hence, using the fact that $q\leq p$ and applying $N-d$ times Lemma \ref{ell1ekiv}, we get
\begin{align}
I_\delta&\lesssim \sum_{\beta\in\N^N}2^{\varrho_0 q|\beta|}\sum_{\nu\geq0}\bigg(\sum_{k\in\N^{N-d}}|\Lambda_{\nu,k+\delta}^\beta|^p\bigg)^{q/p} \nonumber \\
&\leq \sum_{\beta\in\N^N}2^{(\varrho_0-\varrho) q|\beta|}\sup_{\beta\in\N^N}2^{\varrho q|\beta|}\sum_{\nu\geq0}\bigg(\sum_{k\in\N^{N-d}}|\Lambda_{\nu,k+\delta}^\beta|^p\bigg)^{q/p} \nonumber \\
&= K_{\varrho,N,q}\sup_{\beta\in\N^N}2^{\varrho q|\beta|}\sum_{\nu\geq0}\bigg(\sum_{k\in\N^{N-d}}|\Lambda_{\nu,k+\delta}^\beta|^p\bigg)^{q/p} \nonumber \\
&\leq K_{\varrho,N,q}\sup_{\beta\in\N^N}2^{\varrho q|\beta|}\|\lambda^\beta\|_{b_{p,q}}^q. \nonumber
\end{align}
Thus, recalling \eqref{majoration22}, we arrive at
\begin{align*}
\left(\int_{[1,2]^{N-d}}\|f(\cdot,x'')\|_{B_{p,q}^s(\R^d)}^q\mathrm{d}x''\right)^{1/q}\lesssim \|f\|_{B_{p,q}^s(\R^N)}.
\end{align*}
Now, having in mind Remark \ref{R:dommm}, we can reproduce exactly the same proof when $\Psi\not\equiv 1$ with almost no modifications. This completes the proof.
\end{proof}

\section{The case $p<q$}\label{SE:DIFFICILE}

In this section we prove, at a stroke, Theorem \ref{THEOREM} and Theorem \ref{TH:BMO}. As will become clear, the proof of Theorem \ref{TH:BMO} will easily follow from that of Theorem \ref{THEOREM}.

Let us begin with the following more general result:
\begin{theo}\label{THEOREM2}
Let $N\geq2$, $1\leq d<N$, $0< p<q\leq\infty$ and $s>\sigma_p$. Let $\Psi$ be an admissible function. Then, there exists a function $f\in B_{p,q}^{(s,\Psi)}(\R^N)$ such that
\begin{align}
f(\cdot,x'')\notin B_{p,\infty}^{(s,\Psi)}(\R^d) \qquad{\mbox{ for a.e. }}~~x''\in \R^{N-d}. \nonumber
\end{align}
\end{theo}
\begin{proof}
We will essentially follow two steps. \\

\noindent \emph{Step 1: case $d=N-1$.} We will construct a function satisfying the requirements of Theorem \ref{THEOREM2} (and hence of Theorem \ref{THEOREM}) via its subatomic coefficients.

Let $\Psi$ be an admissible function. Let $0<p<q\leq\infty$, $s>\sigma_p$, $M=\lfloor s\rfloor+1$ and $(\lambda_{j,k})_{j,k\geq0}\in b_{p,q}$ be the sequence constructed at Corollary \ref{COR:CTR2}.

Also, we let $\psi\in C_c^\infty(\R^N)$ be a function such that
\begin{align}
\mathrm{supp}(\psi)\subset[-2,2]^N,\,\,~\inf_{z\in [0,1]^N}\psi(z)>0~~~~\,\,\text{ and }\,\,~~~~\sum_{m\in\Z^N}\psi(\cdot-m)\equiv1. \label{CDNpsi}
\end{align}
In addition, we will suppose that $\psi$ has the product structure
\begin{align}
\psi(x)=\psi(x_1)\,...\,\psi(x_N). \label{prodstructure}
\end{align}
Notice that such a $\psi$ always exists.\footnotemark\footnotetext{Here is an example. Let $u(t):=e^{-1/t^2}\mathds{1}_{(0,\infty)}(t)$ (extended by $0$ in $(-\infty,0]$) and let $v(t)=u(1+t)u(1-t)$. Then,
$$ \psi(x):=\prod_{j=1}^N\frac{1}{2}\,\psi_0\left(\frac{x_j}{2}\right)~~~~\,\,\text{ where }\,\,~~~~\psi_0(t)=\frac{v(t)}{v(t-1)+v(t)+v(t+1)}, $$
is a smooth function satisfying \eqref{CDNpsi} and \eqref{prodstructure}.}
Then, we define
\begin{align}
f(x)=\sum_{j\geq0}\sum_{k\geq0}\lambda_{j,k}2^{-j(s-\frac{N}{p})}\Psi(2^{-j})^{-1}\psi(2^j(x_1\!-\!C_Mj))...\psi(2^j(x_{N-1}\!-\!C_Mj))\psi(2^jx_N\!-\!k), \label{DE:f}
\end{align}
where $C_M=2(M+2)$. It follows from Definition \ref{G:quarks} that
$$ \Psi(2^{-j})^{-1}2^{-j(s-\frac{N}{p})}\psi(2^jx-m) \qquad{\mbox{ for }}~~x\in\R^N, $$
with
$$m=(C_M2^jj,\,...\,,\,C_M2^jj,\,k)\in\Z^N, $$
can be interpreted as $(s,p,\Psi)$-$0$-quarks relative to the cube $Q_{j,m}$. Consequently, by Theorem \ref{G:BspqQuark} and Corollary \ref{COR:CTR2}, we have
$$ \|f\|_{B_{p,q}^{(s,\Psi)}(\R^N)}\leq c\hspace{0.1em}\bigg(\sum_{j\geq0}\bigg(\sum_{k\geq0}\lambda_{j,k}^p\bigg)^{q/p}\bigg)^{1/q}<\infty, $$
(modification if $q=\infty$). Therefore, $f\in B_{p,q}^{(s,\Psi)}(\R^N)$. In particular, the sum in the right-hand side of \eqref{DE:f} converges in $L^p(\R^N)$ and is unconditionally convergent for a.e. $x\in\R^N$ (notice the terms involved are all nonnegative) and, by Fubini, $f(\cdot,x_N)$ also converges in $L^p(\R^{N-1})$ for a.e. $x_N\in\R$. Thus, letting
$$ \Lambda_j(x_N):=\sum_{k\geq0}\lambda_{j,k}2^{j/p}\psi(2^jx_N-k), $$
we may rewrite \eqref{DE:f} as
$$ f(x',x_N)=\sum_{j\geq0}\Lambda_j(x_N)\hspace{0.1em}2^{-j(s-\frac{N-1}{p})}\Psi(2^{-j})^{-1}\psi(2^j(x_1-C_Mj))\,...\,\psi(2^j(x_{N-1}-C_Mj)). $$
Notice that assumption \eqref{CDNpsi} implies that there is a $c_0>0$ such that
$$ \psi(2^jx_N-\lfloor2^jx_N\rfloor)\geq c_0>0 \qquad{\mbox{ for all }}~~x_N\in[1,2]~~\text{ and }~~j\geq0. $$
In particular, we have
\begin{align}
\Lambda_j(x_N)\geq c_0\,\lambda_{j,\lfloor 2^jx_N\rfloor}2^{j/p}. \label{loweta}
\end{align}
Now, for all $j\geq0$, we write
\begin{align}
K_j:=\big\{h\in\R^{N-1}:2^{-(j+1)}\leq|h|\leq 2^{-j}\big\}. \label{DE:Kj}
\end{align}
By \cite[Lemma 8.2]{JB} (in fact in \cite{JB} it is implicitly supposed that $1\leq p<\infty$ but the proof still works when $0<p<1$) and \eqref{loweta}, we have
\begin{align}
\sup_{h\in K_j}\|\Delta_h^Mf(\cdot,x_N)\|_{L^p(\R^{N-1})}&\geq c\,2^{-js}\Psi(2^{-j})^{-1}\Lambda_j(x_N) \nonumber \\
&\geq c'\,2^{-js}\Psi(2^{-j})^{-1}\,2^{j/p}\lambda_{j,\lfloor2^jx_N\rfloor}, \label{below}
\end{align}
for any $j\geq0$ and some $c'>0$ independent of $j$. Recall that
$$ \|g\|_{B_{p,\infty}^{(s,\Psi)}(\R^{N-1})}\sim\|g\|_{L^p(\R^{N-1})}+\sup_{j\geq1}\,2^{js}\Psi(2^{-j})\sup_{h\in K_j}\|\Delta_h^Mg\|_{L^p(\R^{N-1})}, $$
is an equivalent quasi-norm on $B_{p,\infty}^{(s,\Psi)}(\R^{N-1})$ (this is a discretized version of Definition \ref{G:BESOV}). This together with \eqref{below} and Corollary \ref{COR:CTR2} gives
$$ \|f(\cdot,x_N)\|_{B_{p,\infty}^{(s,\Psi)}(\R^{N-1})}\gtrsim\,\sup_{j\geq0}\,2^{j/p}\lambda_{j,\lfloor2^jx_N\rfloor}=\infty \qquad{\mbox{ for a.e. }}~~x_N\in[1,2]. $$
Therefore, $f(\cdot,x_N)\notin B_{p,\infty}^{(s,\Psi)}(\R^{N-1})$ for a.e. $x_N\in[1,2]$.

We will show that one can construct a function satisfying the requirements of Theorem \ref{THEOREM2} by considering a weighted sum of translates of the function $f$ constructed above. To this end, we let
$$f_l(x',x_N):=f(x',x_N+l) \qquad{\mbox{ for }}~~l\in\Z,$$
and we define
$$ g:=\sum_{l\in\Z}2^{-|l|}f_l. $$
Then, by the triangle inequality for Besov quasi-norms, we have
$$ \|g\|_{B_{p,q}^{(s,\Psi)}(\R^N)}^\eta\leq\sum_{l\in\Z}2^{-\eta|l|}\|f_l\|_{B_{p,q}^{(s,\Psi)}(\R^N)}^\eta\leq c_\eta\|f\|_{B_{p,q}^{(s,\Psi)}(\R^N)}^\eta<\infty, $$
for some $0<\eta\leq1$. Hence, $g\in B_{p,q}^{(s,\Psi)}(\R^N)$. To complete the proof we need to show that
\begin{align}
g(\cdot,x_N)\notin B_{p,\infty}^{(s,\Psi)}(\R^{N-1}) \qquad{\mbox{ for a.e. }}~~x_N\in\R. \label{CONCLUSION}
\end{align}
Let $m\in\Z$. Then, by the triangle inequality for Besov quasi-norms we have
\begin{align}
2^{-\eta |m|}\|f_m(\cdot,x_N)\|_{B_{p,\infty}^{(s,\Psi)}(\R^{N-1})}^\eta&\leq \|g(\cdot,x_N)\|_{B_{p,\infty}^{(s,\Psi)}(\R^{N-1})}^\eta+\sum_{l\neq m}2^{-\eta|l|}\|f_l(\cdot,x_N)\|_{B_{p,\infty}^{(s,\Psi)}(\R^{N-1})}^\eta \nonumber \\
&\leq \|g(\cdot,x_N)\|_{B_{p,\infty}^{(s,\Psi)}(\R^{N-1})}^\eta+c_\eta\,\sup_{l\neq m}\,\|f_l(\cdot,x_N)\|_{B_{p,\infty}^{(s,\Psi)}(\R^{N-1})}^\eta. \label{triangleF}
\end{align}
Clearly, the left-hand side of \eqref{triangleF} is infinite for a.e. $x_N\in[1-m,2-m]$. Thus, to prove \eqref{CONCLUSION}, one only needs to make sure that the last term on the right-hand side of \eqref{triangleF} is finite for a.e. $x_N\in[1-m,2-m]$. For it, we notice that, by construction, it is necessary to have
\begin{align}
j\geq1~~\,\text{ and }\,~~2^j\leq k<2^{1+j}, \label{CD:la}
\end{align}
for $\lambda_{j,k}\neq0$ to hold. In particular, $\Lambda_0\equiv0$ and $\Lambda_j(x_N)$ consists only in finitely many terms for a.e. $x_N\in\R$.
In addition, by our assumptions on the support of $\psi$, we have $\psi(2^jx_N-k)\neq0$ provided
\begin{align}
\left|x_N-\frac{k}{2^j}\right|\leq2^{1-j}. \label{CDN:supp}
\end{align}
By \eqref{CD:la} and \eqref{CDN:supp}, we deduce that if $x_N\in\R\setminus[1,2]$, then there are only finitely many values of $j\geq1$ such that $\Lambda_j(x_N)\not\equiv0$. In particular,
\begin{align}
f(\cdot,x_N+l)\in B_{p,\infty}^{(s,\Psi)}(\R^{N-1}) \qquad{\mbox{ for a.e. }}~~x_N\in[1,2]~~\text{ and all }~~l\in\Z\setminus\{0\}. \label{extP1}
\end{align}
Moreover, a consequence of \eqref{CD:la} and \eqref{CDN:supp} is that
$$ j\geq1~~\text{ and }~~x_N\in\mathrm{supp}(\Lambda_j)~~\,\Longrightarrow\,~~1-2^{1-j}\leq x_N<2+2^{1-j}. $$
In turn, this implies that the support of $\Lambda_j$ is included in $[0,3]$. Therefore,
\begin{align}
f(\cdot,x_N+l)\equiv0 \qquad{\mbox{ for a.e. }}~~x_N\in[1,2]~~\text{ and all }~~l\in\Z~~\text{ with }~~|l|\geq2. \label{extP2}
\end{align}
Hence, by \eqref{extP1} and \eqref{extP2}, we infer that
$$ \max_{l\neq0}\,\|f_l(\cdot,x_N)\|_{B_{p,\infty}^{(s,\Psi)}(\R^{N-1})}<\infty \qquad{\mbox{ for a.e. }}~~x_N\in[1,2]. $$
In like manner, for every $m\in\Z$, we have
$$ \max_{l\neq m}\,\|f_l(\cdot,x_N)\|_{B_{p,\infty}^{(s,\Psi)}(\R^{N-1})}<\infty \qquad{\mbox{ for a.e. }}~~x_N\in[1-m,2-m]. $$
This proves the theorem for $d=N-1$. \\

\noindent \emph{Step 2: case $1\leq d<N-1$.} By the above, we know that Theorem \ref{THEOREM} holds for any $N\geq 2$ and $d=N-1$. In particular, there exists a function $f\in B_{p,q}^{(s,\Psi)}(\R^{d+1})$ such that $f(\cdot,x_{d+1})\notin B_{p,\infty}^{(s,\Psi)}(\R^{d})$ for a.e. $x_{d+1}\in\R$. Now, pick a function $w\in\mathcal{S}(\R^{N-d-1})$ with $w>0$ on $\R^{N-d-1}$ and set
$$ g(x)=g(x_1,...,x_N)=f(x_1,...,x_d,x_{d+1})w(x_{d+2},...,x_N). $$
It is standard that $g\in L^{\overline{p}}(\R^N)$ where $\overline{p}:=\max\{1,p\}$. Then, letting $M=\lfloor s\rfloor+1$ and using \cite[Formula (16), p.112]{Triebel}, we have that
$$ \sup_{|h|\leq t}\,\|\Delta_h^{2M}g\|_{L^p(\R^N)}\lesssim\|f\|_{L^p(\R^{d+1})}\sup_{|h''|\leq t}\|\Delta_{h''}^Mw\|_{L^p(\R^{N-d-1})}+\|w\|_{L^p(\R^{N-d+1})}\sup_{|h'|\leq t}\|\Delta_{h'}^Mf\|_{L^p(\R^{d+1})}, $$
for any $h=(h',h'')\in\R^N\setminus\{0\}$ with $h'=(h_1,...,h_{d+1})$ and $h''=(h_{d+2},...,h_N)$. In particular, recalling Remark \ref{G:M}, we see that this implies
$$ \|g\|_{B_{p,q}^{(s,\Psi)}(\R^N)}\lesssim \|f\|_{L^p(\R^{d+1})}\|w\|_{B_{p,q}^{(s,\Psi)}(\R^{N-d-1})}+\|w\|_{L^p(\R^{N-d-1})}\|f\|_{B_{p,q}^{(s,\Psi)}(\R^{d+1})}. $$
Hence, $g\in B_{p,q}^{(s,\Psi)}(\R^N)$. Moreover, it is easily seen that
$$g(\cdot,x_{d+1},...,x_N)=f(\cdot,x_{d+1})w(x_{d+2},...,x_N)\notin B_{p,\infty}^{(s,\Psi)}(\R^d), $$
for a.e. $(x_{d+1},...,x_N)\in\R^{N-d}$. This completes the proof.
\end{proof}
The function we have constructed above (in the proof of Theorem \ref{THEOREM2}) turns out to verify the conclusion of Theorem \ref{TH:BMO}.
\begin{proof}[Proof of Theorem \ref{TH:BMO}]
For simplicity, we outline the proof for $N=2$ and $d=1$ only (the general case follows from the same arguments as above). Let $f$ be the function constructed in the proof of Theorem \ref{THEOREM2} with $\Psi\equiv1$, namely
$$ f(x_1,x_2):=\sum_{j\geq0}\Lambda_j(x_2)\hspace{0.1em}2^{-j(s-\frac{1}{p})}\psi(2^j(x_1-C_Mj)), $$
with
$$ \Lambda_j(x_2):=\sum_{k\geq0}\lambda_{j,k}2^{j/p}\psi(2^jx_2-k), $$
where $\psi$, $C_M$ and $(\lambda_{j,k})_{j,k\geq0}$ are as in the proof of Theorem \ref{THEOREM2}. Clearly,
$$\|f\|_{B_{p,q}^s(\R^2)}\leq c\hspace{0.1em}\bigg(\sum_{j\geq0}\bigg(\sum_{k\geq0}\lambda_{j,k}^p\bigg)^{q/p}\bigg)^{1/q}<\infty.$$
Hence, $f\in B_{p,q}^s(\R^2)$. We now distinguish the cases $sp>1$, $sp=1$ and $sp<1$. \\

\noindent \emph{Step 1: case} $sp>1$. This case works as in Theorem \ref{THEOREM2}. Indeed, by the supports of the functions involved, we have for a.e. $x_2\in[1,2]$,
$$ \|f(\cdot,x_2)\|_{C^{s-\frac{1}{p}}(\R)}\sim\,\sup_{j\geq0}\,2^{j(s-\frac{1}{p})}\sup_{h\in K_j}\|\Delta_h^Mf(\cdot,x_2)\|_{L^\infty(\R)}\gtrsim \,\sup_{j\geq0}\,2^{j/p}\lambda_{j,\lfloor2^jx_2\rfloor}=\infty, $$
where $K_j$ is given by \eqref{DE:Kj}. We may now conclude as in the proof of Theorem \ref{THEOREM2}.  \\

\noindent \emph{Step 2: case} $sp=1$. It suffices to notice that, for any $k\geq0$, we have
\begin{align}
\|f(\cdot,x_2)\|&_{\mathrm{BMO}(\R)}\geq \fint_{C_Mk-2^{-k}}^{C_Mk+2^{-k}}\left|\fint_{C_Mk-2^{-k}}^{C_Mk+2^{-k}}\big[f(x,x_2)-f(z,x_2)\big]\mathrm{d}z\right|\mathrm{d}x \nonumber \\
&=\fint_{C_Mk-2^{-k}}^{C_Mk+2^{-k}}\left|\sum_{j\geq0}\Lambda_j(x_2)\fint_{C_Mk-2^{-k}}^{C_Mk+2^{-k}}\big[\psi(2^j(x-C_Mj))-\psi(2^j(z-C_Mj))\big]\mathrm{d}z\right|\mathrm{d}x. \nonumber
\end{align}
Hence, by the support of the functions involved we deduce that
\begin{align}
\|f(\cdot,x_2)\|_{\mathrm{BMO}(\R)}&\geq \Lambda_k(x_2)\fint_{C_Mk-2^{-k}}^{C_Mk+2^{-k}}\left|\fint_{C_Mk-2^{-k}}^{C_Mk+2^{-k}}\big[\psi(2^k(x-C_Mk))-\psi(2^k(z-C_Mk))\big]\mathrm{d}z\right|\mathrm{d}x \nonumber \\
&=\Lambda_k(x_2)\fint_{-1}^{1}\left|\fint_{-1}^{1}\big[\psi(x)-\psi(z)\big]\mathrm{d}z\right|\mathrm{d}x \geq c'\Lambda_k(x_2). \nonumber
\end{align}
Therefore, we have
$$ \|f(\cdot,x_2)\|_{\mathrm{BMO}(\R)}\gtrsim \,\sup_{j\geq0}\,\Lambda_j(x_2)=\sup_{j\geq0}\,2^{j/p}\lambda_{j,\lfloor2^jx_2\rfloor}=\infty \quad{\mbox{ for a.e. }}x_2\in[1,2]. $$
Thus, we may again conclude as in the proof of Theorem \ref{THEOREM2}. \\

\noindent \emph{Step 3: case} $sp<1$. Define $r:=\frac{p}{1-sp}$ and rewrite $f$ as
$$ f(x_1,x_2):=\sum_{j\geq0}c_j(x_2)\,2^{j/r}f_j(x_1), $$
where we have set $f_j(x_1):=\psi(2^j(x_1-C_Mj))$ and
$$ c_j(x_2):=2^{-j(s-\frac{2}{p})}\,2^{-j/r}\sum_{k\geq0}^{~}\lambda_{j,k}\psi(2^jx_2-k). $$
Since the $f_j$'s have mutually disjoint support we find that
$$ f(\cdot,x_2)^*(t)\geq c_j(x_2)\,2^{j/r}f_j^*(t) \quad{\mbox{ for any }}t\geq0\text{ and }j\geq0. $$
Moreover, it is easy to see that $f_j^*(t)=\psi^*(2^jt)$.
In turn, this implies that
\begin{align}
\|f(\cdot,x_2)\|_{L^{r,\infty}(\R)}&\geq c_j(x_2)\,2^{j/r}\sup_{t>0}\, t^{1/r}\psi^*(2^jt) =c_j(x_2)\|\psi\|_{L^{r,\infty}(\R)} \gtrsim 2^{j/p}\lambda_{j,\lfloor 2^j x_2\rfloor}. \nonumber
\end{align}
Hence, for a.e. $x_2\in[1,2]$,
$$ \|f(\cdot,x_2)\|_{L^{r,\infty}(\R)}\gtrsim \,\sup_{j\geq0}\,2^{j/p}\lambda_{j,\lfloor2^jx_2\rfloor}=\infty.$$
This completes the proof.
\end{proof}

\section{Characterization of restrictions of Besov functions}\label{SE:WEIGHT}

In this section, we prove that Besov spaces of generalized smoothness are the natural scale in which to look for restrictions of Besov functions. More precisely, we will prove Theorems \ref{TH:WEIGHT} and \ref{TH:WEIGHT33}. We present several results, with different assumptions and different controls on the norm of $f(\cdot,x'')$.

Let us begin with the following
\begin{theo}\label{TH2:WEIGHT}
Let $N\geq2$, $1\leq d<N$, $0<p<q\leq\infty$ and let $s>\sigma_p$. Let $K\subset\R^{N-d}$ be a compact set. Let $\Phi$ and $\Psi$ be two admissible functions such that
\begin{align}
\sum_{j\geq0}\Phi(2^{-j})^{-p}\Psi(2^{-j})^p<\infty. \label{H:WEIGHT222}
\end{align}
Let $f\in B_{p,q}^{(s,\Phi)}(\R^N)$. Then, there exists a constant $C>0$ such that
\begin{align*}
\|f(\cdot,x'')\|_{B_{p,q}^{(s,\Psi)}(\R^d)}\leq C\|f\|_{B_{p,q}^{(s,\Phi)}(\R^N)} \qquad{\mbox{ for a.e. }}~~x''\in K. 
\end{align*}
Moreover, the constant $C$ is independent of $x''$ but may depend on $f$, $K$, $N$, $d$, $p$, $q$, $\Phi$ and $\Psi$.
\end{theo}
\begin{proof}
Let us first suppose that $q<\infty$ and that $\Phi\equiv1$. Without loss of generality we may take $K=[1,2]^{N-d}$, the general case being only a matter of scaling. Here again, we use the following short notation
$$\lfloor 2^\nu x''\rfloor=(\lfloor 2^\nu x_{d+1}\rfloor,...\,,\lfloor2^\nu x_N\rfloor)\in\N^{N-d}.$$
Let $f\in B_{p,q}^s(\R^N)$ and write its subatomic decomposition as
$$ f(x)=\sum_{\beta\in\N^N}\sum_{\nu\geq0}\sum_{m\in\Z^N}\lambda_{\nu,m}^\beta(\beta\mathrm{qu})_{\nu,m}(x), $$
where the $(s,p)$-$\beta$-quarks $(\beta\mathrm{qu})_{\nu,m}$ are as in Section \ref{SE:ESTIMATE} and $\lambda_{\nu,m}^\beta=\lambda_{\nu,m}^\beta(f)$ is the optimal subatomic decomposition of $f$, i.e. such that
\begin{align}
\|f\|_{B_{p,q}^s(\R^N)}\sim\sup_{\beta\in\N^N}2^{\varrho|\beta|}\|\lambda^\beta\|_{b_{p,q}}. \label{WEI:HYPOF}
\end{align}
Let $\eps>0$ be small. Rewriting $f$ as in the discussion at Section \ref{SE:ESTIMATE} and using Lemma \ref{dommm} together with Remark \ref{R:dommm} we have
\begin{align}
\|f(\cdot,x'')\|_{B_{p,q}^{(s,\Psi)}(\R^d)}\lesssim \sum_{\delta\in\mathcal{D}}\tilde{J}_{p,q}^{\varrho-\eps,\delta}(\lambda,x''), \label{G:sub}
\end{align}
where
\begin{align} \tilde{J}_{p,q}^{\varrho-\eps,\delta}(\lambda,x''):=\sup_{\beta\in\N^N}2^{(\varrho-\eps)|\beta|}\bigg(\sum_{\nu\geq0}\bigg(2^{\nu(N-d)}\Psi(2^{-\nu})^p\sum_{m'\in\Z^d}|\lambda_{\nu,m',\lfloor2^\nu x''\rfloor+\delta}^\beta|^p\bigg)^{q/p}\bigg)^{1/q}. \label{leJi}
\end{align}
By Lemma \ref{LE:tech}, we know that for any positive sequence $(\alpha_\nu)_{\nu\geq0}\in\ell^1(\N)$ there is a constant $C=C(\lambda,\alpha,N,d)>0$ such that
\begin{align*}
2^{\nu(N-d)}\sum_{m'\in\Z^d}|\lambda_{\nu,m',\lfloor 2^\nu x''\rfloor+\delta}^\beta|^p\leq C\frac{\max\{1,|\beta|^{N-d+1}\}}{\alpha_\nu}\sum_{m\in\Z^N}|\lambda_{\nu,m}^\beta|^p, 
\end{align*}
for a.e. $x''\in[1,2]^{N-d}$ and any $(\nu,\beta)\in\N\times\N^N$.
In particular, we have
$$ \tilde{J}_{p,q}^{\varrho-\eps,\delta}(\lambda,x'')\lesssim \sup_{\beta\in\N^N}2^{(\varrho-\eps)|\beta|}\max\{1,|\beta|^{\frac{N-d+1}{p}}\}\bigg(\sum_{\nu\geq0}\bigg(\frac{\Psi(2^{-\nu})^p}{\alpha_\nu}\sum_{m\in\Z^N}|\lambda_{\nu,m}^\beta|^p\bigg)^{q/p}\bigg)^{1/q}. $$
Now, by assumption \eqref{H:WEIGHT222}, we can choose $\alpha_\nu=\Psi(2^{-\nu})^p$. Therefore, recalling \eqref{G:sub}, we have
\begin{align}
\|f(\cdot,x'')\|_{B_{p,q}^{(s,\Psi)}(\R^d)}&\lesssim \sup_{\beta\in\N^N}2^{(\varrho-\eps)|\beta|}\max\{1,|\beta|^{\frac{N-d+1}{p}}\}\bigg(\sum_{\nu\geq0}\bigg(\sum_{m\in\Z^N}|\lambda_{\nu,m}^\beta|^p\bigg)^{q/p}\bigg)^{1/q} \nonumber \\
&\leq \sup_{\beta\in\N^N}2^{\varrho|\beta|}\bigg(\sum_{\nu\geq0}\bigg(\sum_{m\in\Z^N}|\lambda_{\nu,m}^\beta|^p\bigg)^{q/p}\bigg)^{1/q}. \nonumber
\end{align}
Finally, recalling \eqref{WEI:HYPOF}, we have
\begin{align*}
\|f(\cdot,x'')\|_{B_{p,q}^{(s,\Psi)}(\R^d)}\lesssim \|f\|_{B_{p,q}^s(\R^N)} \qquad{\mbox{ for a.e. }}~~x''\in[1,2]^{N-d}.
\end{align*}
The proof when $q=\infty$ and/or $\Phi\not\equiv1$ is similar (recall Remark \ref{R:dommm}). It this latter case, one only have to adjust the $(s,p)$-$\beta$-quarks by a factor of $\Phi(2^{-\nu})^{-1}$ and to replace $\lambda_{\nu,m}^\beta$ by $\eta_{\nu,m}^\beta$, the optimal decomposition of $f\in B_{p,q}^{(s,\Phi)}(\R^N)$ along these $(s,p,\Phi)$-$\beta$-quarks. Then, as in Remark \ref{R:dommm}, it suffices to replace $\Psi(2^{-\nu})\lambda_{\nu,m}^\beta$ in the estimates \eqref{G:sub} and \eqref{leJi} of $\|f(\cdot,x'')\|_{B_{p,q}^{(s,\Psi)}(\R^d)}$ by $\Psi(2^{-\nu})/\Phi(2^{-\nu})\,\eta_{\nu,m}^\beta$ and the same proof yields the desired conclusion.
\end{proof}
We carry on with the following generalization of Theorem \ref{TH:WEIGHT}.
\begin{theo}\label{TH3:WEIGHT}
Let $N\geq 3$, $1\leq d<N$, $0<r\leq p<q\leq \infty$, $s>\sigma_p$ and let $\chi=\frac{qr}{q-r}$ (resp. $\chi=r$ if $q=\infty$). Let $\Phi$ and $\Psi$ be two admissible functions such that
\begin{align}
\sum_{j\geq0}\Phi(2^{-j})^{-\chi}\Psi(2^{-j})^\chi<\infty. \label{H:WEIGHT55}
\end{align}
Let $K\subset\R^{N-d}$ be a compact set and suppose that $f\in B_{p,q}^{(s,\Phi)}(\R^N)$. Then,
$$ \left(\int_K\|f(\cdot,x'')\|_{B_{p,r}^{(s,\Psi)}(\R^d)}^r\mathrm{d}x''\right)^{1/r}\leq C\|f\|_{B_{p,q}^{(s,\Phi)}(\R^N)}, $$
for some constant $C=C(K,N,d,p,q,\Psi)>0$.
\end{theo}
\begin{proof}
By H\"older's inequality, the definition of the norms involved and our assumptions on $s$, $p$, $q$, $r$, $\chi$, we see that \eqref{H:WEIGHT55} implies that $B_{p,q}^{(s,\Phi)}(\R^N)\subset B_{p,r}^{(s,\Psi)}(\R^N)$ continuously, i.e.
$$\|f\|_{B_{p,r}^{(s,\Psi)}(\R^N)}\leq \bigg(\sum_{j\geq0}\Phi(2^{-j})^{-\chi}\Psi(2^{-j})^\chi\bigg)^{1/\chi}\|f\|_{B_{p,q}^{(s,\Phi)}(\R^N)}.$$
Using now Proposition \ref{caracBesov2}, we obtain the desired conclusion.
\end{proof}
We now prove Theorem \ref{TH:WEIGHT33}.
\begin{proof}[Proof of Theorem \ref{TH:WEIGHT33}]
The proof works exactly as in Theorem \ref{THEOREM} and, here again, it suffices to prove the result for $N\geq2$, $d=N-1$. We prove the case $N=2$ only but the general case $N\geq2$ is similar. Let $(\lambda_{j,k})_{j,k\geq0}$ be the sequence constructed at Lemma \ref{LE:CTR:WEIGHT3}. Let $M\in\N^*$ with $s<M$. We consider the following function
\begin{align*}
f(x_1,x_2)=\sum_{j\geq0}\sum_{k\geq0}\lambda_{j,k}2^{-j(s-\frac{2}{p})}\psi(2^j(x_1-C_Mj))\psi(2^jx_2-k), 
\end{align*}
where $\psi$ and $C_M$ are as in the proof of Theorem \ref{THEOREM}. There, we have shown that
$$ \|f\|_{B_{p,q}^s(\R^2)}\leq c\|\lambda\|_{b_{p,q}}, $$
(thus implying $f\in B_{p,q}^s(\R^2)$) and
\begin{align*}
\sup_{h\in K_j}\,\|\Delta_h^Mf(\cdot,x_2)\|_{L^p(\R)}\geq c\,2^{-js}2^{j/p}\lambda_{j,\lfloor2^jx_2\rfloor}, 
\end{align*}
for any $j\geq0$ and a.e. $x_2\in[1,2]$, where $K_j$ is as in \eqref{DE:Kj}. From this it follows that
\begin{align*}
2^{js}\Psi(2^{-j})\sup_{h\in K_j}\,\|\Delta_h^Mf(\cdot,x_2)\|_{L^p(\R)}&\geq c\,\Psi(2^{-j})2^{j/p}\lambda_{j,\lfloor2^jx_2\rfloor}.
\end{align*}
Now, by Lemma \ref{LE:CTR:WEIGHT3} and since
\begin{align*}
\|g\|_{B_{p,q}^{(s,\Psi)}(\R)}&\sim \|g\|_{L^p(\R)}+\bigg(\sum_{j\geq0}\,2^{jsq}\Psi(2^{-j})^q\sup_{h\in K_j}\,\|\Delta_h^Mg\|_{L^p(\R)}^q\bigg)^{1/q},
\end{align*}
is an equivalent quasi-norm on $B_{p,q}^{(s,\Psi)}(\R)$, we have $f\in B_{p,q}^s(\R^2)$ and $f(\cdot,x_2)\notin B_{p,q}^{(s,\Psi)}(\R)$ for a.e. $x_2\in[1,2]$. Then, arguing exactly as in Theorem \ref{THEOREM}, we obtain a function satisfying the requirements of Theorem \ref{TH:WEIGHT33}. This completes the proof.
\end{proof}
\begin{remark}\label{RE:BGene}
This can be generalized in the spirit of Theorem \ref{TH2:WEIGHT}. Indeed, repeating the arguments of the proof of Theorem \ref{THEOREM2}, we can prove that if \eqref{H:WEIGHT55} is violated, then there is a function $f\in B_{p,q}^{(s,\Phi)}(\R^N)$ such that $f(\cdot,x'')\notin B_{p,q}^{(s,\Psi)}(\R^d)$ for a.e. $x''\in\R^{N-d}$.
\end{remark}

\section*{Acknowledgments} The author would like to thank \emph{Petru Mironescu} who attracted his attention to the problem and for useful discussions. The author is also grateful to \emph{J\'er\^ome Coville}, \emph{Fran\c{c}ois Hamel} and \emph{Enrico Valdinoci} for their encouragements. The author warmly thanks the anonymous referee for careful reading of the manuscript and useful remarks that
led to the improvement of the presentation. This project has been supported by the French National Research Agency (ANR) in the framework of the ANR NONLOCAL project (ANR-14-CE25-0013).

\vspace{2mm}


\begin{thebibliography}{12}
\setlinespacing{0.95}
\frenchspacing

\bibitem{Almeida} {\sc A. Almeida, A. Caetano}:
Real interpolation of generalized Besov-Hardy spaces and applications.
{\em J. Fourier Anal. Appl.}, 17, pp.691-719, (2011).

\bibitem{Ansorena2} {\sc J. L. Ansorena, O. Blasco}:
Atomic decomposition of weighted Besov spaces.
{\em J. London Math. Soc.}, 53(1), pp.127-140, (1996).

\bibitem{Ansorena} {\sc J. L. Ansorena, O. Blasco}:
Characterization of weighted Besov spaces.
{\em Math. Nachr.}, 171(1), pp.5-17, (2006).

\bibitem{Ash} {\sc J. M. Ash}:
Neither a worst convergent nor a best divergent series exists.
{\em College Math. J.}, 28, pp.296-297, (1997).

\bibitem{Aubry} {\sc J.-M. Aubry, D. Maman, S. Seuret}:
Local behavior of traces of Besov functions: prevalent results.
{\em J. Func. Anal.}, 264(3), pp.631-660, (2013).

\bibitem{BBM} {\sc J. Bourgain, H. Brezis, P. Mironescu}:
Lifting in Sobolev spaces.
{\em J. Anal. Math.}, 80, pp.37-86, (2000).

\bibitem{Bingham} {\sc N. H. Bingham, C. M. Goldie, J. L. Teugels}:
Regular variation.
{\em Cambridge Univ. Press}, (1987).

\bibitem{JB} {\sc J. Brasseur}:
A Bourgain-Brezis-Mironescu characterization of higher order Besov-Nikol'skii spaces.
{\em Ann. Inst. Fourier} (to appear), arXiv:1610.05162, (2017).

\bibitem{Bricchi} {\sc M. Bricchi}:
Complements and results on $h$-sets.
{\em Function Spaces, Differential Operators and Nonlinear Analysis (The Hans Triebel Anniversary Volume)}, Birk\"auser, pp.219-229, (2003).

\bibitem{Bricchi2} {\sc M. Bricchi}:
Tailored Besov spaces and $h$-sets.
{\em Math. Nachr.}, 263-264, pp.36-52, (2004).

\bibitem{Caetano} {\sc A. M. Caetano, S. D. Moura}:
Local growth envelopes of spaces of generalized smoothness: the sub-critical case.
{\em Math. Nachr.}, 273(1), pp.43-57, (2004).

\bibitem{Caetano2} {\sc A. M. Caetano, S. D. Moura}:
Local growth envelopes of spaces of generalized smoothness: the critical case.
{\em Math. Inequal. Appl.}, 7(4), pp.573-606, (2004).


\bibitem{Cobos} {\sc F. Cobos, L. M. Fern\'andez-Cabrera, H. Triebel}:
Abstract and concrete logarithmic interpolation spaces.
{\em J. London Math. Soc.}, 70(2), pp.231-243, (2004).

\bibitem{Edmunds} {\sc D. Edmunds, H. Triebel}:
Spectral theory for isotropic fractal drums.
{\em C. R. Acad. Sci. Paris}, 326(11), pp.1269-1274, (1998).

\bibitem{ET} {\sc D. Edmunds, H. Triebel}:
Eigenfrequencies of isotropic fractal drums.
{\em Operator Theory: Advances and Applications}, 110, pp.81-102, (1999).

\bibitem{Farkas} {\sc W. Farkas, H. G. Leopold}:
Characterization of function spaces of generalized smoothness.
{\em Annali di Matematica Pura ed Applicata}, 185(1), pp.1-62, (2006).



\bibitem{Gurka} {\sc P. Gurka, B. Opic}:
Sharp embeddings of Besov spaces with logarithmic smoothness.
{\em Rev. Mat. Complutense}, 18, pp.81-110, (2005).

\bibitem{Gurka2} {\sc P. Gurka, B. Opic}:
Sharp embeddings of Besov-type spaces.
{\em J. Comput. Appl. Math.}, 208, pp.235-269, (2007).

\bibitem{Haroske} {\sc D. Haroske, S. D. Moura}:
Continuity envelopes of spaces of generalised smoothness, entropy and approximation numbers.
{\em J. Approx. Theory}, 128, pp.151-174, (2004).

\bibitem{Izuki} {\sc M. Izuki, Y. Sawano, H. Tanaka}:
Weighted Besov-Morrey spaces and Triebel-Lizorkin spaces.
{\em Harmonic analysis and nonlinear partial differential equations},
21C60, RIMS K\^oky\^uroku Bessatsu, B22, Res. Inst. Math. Sci. (RIMS), Kyoto, (2010).

\bibitem{Jacob} {\sc N. Jacob}:
Pseudo Differential Operators $\&$ Markov Processes. Volume III: Markov processes and applications.
{\em Imperial College Press}, (2005).

\bibitem{Jaffard} {\sc S. Jaffard}:
Th\'eor\`emes de trace et "dimensions n\'egatives" (in French).
{\em C. R. Acad. Sci. Paris}, 320(4), pp.409-413, (1995).


\bibitem{Kalton} {\sc N. J. Kalton, N. T. Peck, J. W. Roberts}:
An $F$-space Sampler.
{\em London Math. Soc.}, Cambridge Univ. Press, Lecture Notes vol. 89, (1985).

\bibitem{Kalyabin} {\sc G. A. Kalyabin, P. I. Lizorkin}:
Spaces of functions of generalized smoothness.
{\em Math. Nachr.}, 133, pp.7-32, (1987).

\bibitem{Karamata1} {\sc J. Karamata}:
\"Uber einen Konvergenzsatz des Herrn Knopp (in German).
{\em Math. Zeit.}, 40(1), pp.421-425, (1935).

\bibitem{Karamata2} {\sc J. Karamata}:
Th\'eor\`emes sur la sommabilit\'e exponentielle et d'autres sommabilit\'es s'y rattachant (in French).
{\em Mathematica} (Cluj), 9, pp. 164-178, (1935).

\bibitem{Knopova} {\sc V. Knopova, M. Z\"ahle}:
Spaces of generalized smoothness on $h$-sets and related Dirichlet forms.
{\em Studia Math.}, 174, pp.277-308, (2006).

\bibitem{Kudry} {\sc L. D. Kudryavtsev, S. M. Nikol'skii}:
Spaces of differentiable functions of several variables and imbedding theorem. Analysis III, Spaces of differentiable functions.
{\em Encycl. of Math. Sciences, Heidelberg: Springer}, 26, pp.4-140, (1990).

\bibitem{Leopold1} {\sc H.-G. Leopold}:
On Besov spaces of variable order of differentiation.
{\em Z. Anal. Anwendungen}, 8(1), pp.69-82, (1989).

\bibitem{Leopold} {\sc H.-G. Leopold, E. Schroh}:
Trace theorems for Sobolev spaces of variable order of differentiation.
{\em Math. Nachr.}, 179, pp.223-245, (1996).


\bibitem{LeopoldF} {\sc H.-G. Leopold}:
Embeddings and entropy numbers in Besov spaces of generalized smoothness.
{\em Function Spaces: the fifth conference, Lecture notes in pure and applied math.} (ed. H. Hudzig and L. Skrzypcazk) Marcel Dekker, 213, pp.323-336, (2000).

\bibitem{Lizorkin} {\sc P. I. Lizorkin}:
Spaces of generalized smoothness.
{\em Mir, Moscow}, pp.381-415, (1986). Appendix to Russian ed. of H. Triebel. Theory of function spaces. Birkh\"auser, Basel, (1983).

\bibitem{Petru} {\sc P. Mironescu}:
Personal communication.
(2015).

\bibitem{MRS} {\sc P. Mironescu, E. Russ, Y. Sire}:
Lifting in Besov spaces.
{\em Preprint}, hal-01517735, (2017).

\bibitem{Mou} {\sc S. D. Moura}:
Function spaces of generalized smoothness.
{\em Dissertationes Math.}, 398, pp.1-88, (2001).

\bibitem{MNS} {\sc S. D. Moura, J. S. Neves, C. Schneider}:
Spaces of generalized smoothness in the critical case: optimal embeddings, continuity envelopes and approximation numbers.
{\em J. Approx. Theory}, 187, pp.82-117, (2014).


\bibitem{Nakamura} {\sc S. Nakamura, T. Noi, Y. Sawano}:
Generalized Morrey spaces and trace operator.
{\em Science China Mathematics}, 59(2), pp.281-336, (2016).

\bibitem{JPeetre} {\sc J. Peetre}:
Espaces d'interpolation et th\'eor\`eme de Soboleff (in French).
{\em Ann. Inst. Fourier}, 16(1), pp.279-317, (1966).


\bibitem{Schneider} {\sc R. Schneider, O. Reichmann, C. Schwab}:
Wavelet solution of variable order pseudodifferential equations.
{\em C. Calcolo}, 47(2), pp. 65-101, (2010).


\bibitem{Triebel} {\sc H. Triebel}:
Theory of Function Spaces.
{\em Birkh\"auser, Monographs in Mathematics}, (1983).

\bibitem{TriebSpectra} {\sc H. Triebel}:
Fractals and Spectra.
{\em Birkh\"auser, Monographs in Mathematics}, (1997).

\bibitem{Trieb} {\sc H. Triebel}:
The Structure of Functions.
{\em Birkh\"auser, Monographs in Mathematics}, (2001).

\bibitem{Trieb06} {\sc H. Triebel}:
Theory of Function Spaces III.
{\em Birkh\"auser, Monographs in Mathematics}, (2006).

\bibitem{Trieb10} {\sc H. Triebel}:
Bases in function spaces, sampling, discrepancy, numerical integration.
{\em EMS Publ. House, Z\"urich}, (2010).


\end{thebibliography}
\end{document}